\documentclass[reqno, a4paper]{amsart}
\usepackage{amssymb, amscd, amsfonts,  amsthm}
\usepackage[mathscr]{eucal}
\usepackage{graphicx}

\newtheorem{theorem}{Theorem}[section]
\newtheorem{lemma}[theorem]{Lemma}
\newtheorem{corollary}[theorem]{Corollary}
\newtheorem{proposition}[theorem]{Proposition}

\theoremstyle{definition}

\newtheorem{remark}[theorem]{Remark}

\newcommand{\remove}[1]{}

\DeclareMathOperator{\cl}{\mathrm{cl}}
\DeclareMathOperator{\pos}{\mathrm{pos}}
\DeclareMathOperator{\aff}{\mathrm{aff}}
\DeclareMathOperator{\conv}{\mathrm{conv}} 
\DeclareMathOperator{\den}{\mathrm{den}}

\newcommand{\GLnZ}{\mathsf{GL}(n,\mathbb{Z})\ltimes \mathbb{Z}^n}
\newcommand{\GLnQ}{\mathsf{GL}(n,\mathbb Q)\ltimes \mathbb Q^n}
\newcommand{\GLtwoZ}{\mathsf{GL}(2,\mathbb{Z})\ltimes \mathbb{Z}^2}
\newcommand{\GLtwoQ}{\mathsf{GL}(2,\mathbb Q)\ltimes \mathbb Q^2}

 \title[Basic geometry of the affine group over $\mathbb Z$]
{Basic geometry of the affine group over $\mathbb Z$}

\author{Daniele Mundici }

\address[D. Mundici]{Department of
Mathematics and Computer Science  ``Ulisse Dini'' 
University of Florence\\
Viale Morgagni 67/a \\
I-50134 Florence \\
Italy}
\email{ mundici@math.unifi.it }

\begin{document}


\begin{abstract} 
The subject matter of this paper is  the geometry of
the affine group  over the integers,
 $\mathsf{GL}(n,\mathbb{Z})\ltimes \mathbb{Z}^n$.    Turing-computable complete  
$\mathsf{GL}(n,\mathbb{Z})\ltimes \mathbb{Z}^n$-orbit invariants
are constructed  for   angles, segments,  triangles and ellipses.
In rational affine 
$\mathsf{GL}(n,\mathbb Q)\ltimes \mathbb Q^n$-geometry, ellipses
are classified by the Clifford--Hasse--Witt invariant, via 
the Hasse-Minkowski theorem.
We classify ellipses in  
$\mathsf{GL}(n,\mathbb{Z})\ltimes \mathbb{Z}^n$-geometry   
combining  results  by
Apollonius of Perga and Pappus of Alexandria
with the Hirzebruch-Jung continued fraction algorithm and 
the Morelli-W\l odarczyk solution of the weak
Oda conjecture on the factorization of toric varieties.
We then consider  {\it rational polyhedra},
i.e., finite unions of simplexes in  $\mathbb R^n$  with  rational vertices.  Markov's unrecognizability theorem for combinatorial manifolds states
the undecidability of  the problem whether
two rational polyhedra $P$ and $P'$  are  continuously  
$\mathsf{GL}(n,\mathbb Q)\ltimes \mathbb Q^n$-equidissectable.  
The same problem for the continuous 
 $\mathsf{GL}(n,\mathbb{Z})\ltimes \mathbb{Z}^n$-equi\-dis\-sect\-ability 
 of  $P$ and $P'$  is open. 
We prove the  decidability of the problem whether two
rational polyhedra  $P,Q$   in $\mathbb R^n$ have the same 
$\mathsf{GL}(n,\mathbb{Z})\ltimes \mathbb{Z}^n$-orbit.  
\end{abstract}

\maketitle

\section{Introduction}
\label{section:introduction}
\noindent In Klein's 1872
 inaugural lecture at the University of Erlangen
 one finds the following programmatic statement\footnote{See   
page 219  in his paper 
``A comparative review of recent researches in geometry'', 
 Bull. New York Math. Soc.,  2(10) (1893) 215--249, 
https://projecteuclid.org/euclid.bams/1183407629}:

\begin{quote}
{\small 
As a generalization of geometry arises then the following comprehensive problem
$[\dots]$:
{\it Given a manifoldness and  a group of transformations of the same;  to develop the theory of invariants relating to that group.}
}
\end{quote}
In the spirit of Klein's program, in the first part of this paper 
we construct   Turing-computable complete  invariants of angles, 
segments,  triangles and ellipses  in the geometry of the affine group
  over the integers, $\GLnZ$.
    
    Our starting point is the following  problem,
for arbitrary   $X,X'\subseteq \mathbb R^n$:
\begin{equation}
\label{equation:theproblem}
 \mbox{Does there exist a map
 $\gamma\in \GLnZ$ of  
   $X$ onto $X'$ ?}
\end{equation}

\noindent
Otherwise stated: Do  $X$ and $X'$ have the same $\GLnZ$-orbit ?
By  a ``decision method''   for  this problem 
we understand 
a  Turing machine  $\mathcal M$ which, over 
any  input $X,X'$   decides whether $X$
and $X'$  have  the same orbit.
$X$ and $X'$   must be
effectively presented to $\mathcal M$ as  finite strings of 
 symbols.
  Thus, e.g.,  if $X$ is a triangle, we will assume
that  its vertices have rational coordinates, and
  $X$ is  presented
 to $\mathcal M$ via the list of its vertices.  If
 $X$ is an ellipse,  $X$  is  understood as 
the zeroset   $Z(\phi)$  in $\mathbb R^2$ 
of a quadratic  polynomial  $\phi(x,y)$ with rational coefficients, 
and  is  presented to $\mathcal M$ 
via the list of coefficients of
  $\phi.$ 
  
In Theorem 
\ref{theorem:triangle}
we equip rational triangles with
a  (Turing-) computable complete
$\GLnZ$-orbit invariant.
Basic constituents of  our side-angle-side  invariant are the  
 invariants introduced in 
  \cite{cabmun-etds} 
  for affine spaces in $\GLnZ$-geometry,  (Theorem
   \ref{theorem:classification}).
 Section \ref{section:cd} is devoted to showing
 the computability of these invariants. 
 Further  main constituents
 are the 
 $\GLnZ$-orbit invariants for angles and segments
 constructed   in Theorems 
\ref{theorem:angle} and \ref{theorem:quadruple}.
It follows  that 
Problem \eqref{equation:theproblem} is decidable
for segments, angles and  triangles in $\GLnZ$-geometry.
 
In Section \ref{section:ellipses}
we construct computable complete invariants for ellipses:
In euclidean 
 geometry, ellipses  are
classified by the lengths of their major and minor axes.
In  $\GLtwoQ$-geometry,
 the Hasse-Minkowski theorem  classifies rational
  ellipses  by their 
(Clifford-Hasse-Witt)
invariants,  \cite[1.1]{aue},  \cite[\S 5]{gen},  \cite[\S 4]{kit}.
Let   $\mathcal E$ denote the set of 
 rational ellipses  $E\subseteq \mathbb R^2$  
 containing a rational point. 
As is well known, (see, e.g.,
  \cite{crerus},  \cite{sim}),
 from  the input  rational coefficients of $\phi$  it is
 decidable whether the zeroset of $\phi$ is an
  ellipse $E\in \mathcal E$.
If this is the case,  $E$ contains a dense set of rational points.  
In Theorem \ref{theorem:ellipses} a
 finite set of  invariants
  is computed  from $\phi$, in such a  way that  
a rational ellipse $E'$ has the same  
$\GLtwoZ$-orbit  of $E$ iff
$E$ and $E'$ have the same invariants. 
It follows  that 
Problem \eqref{equation:theproblem} is decidable
for ellipses.

For our constructions in this paper 
we combine results on conjugate diameters  by
Apollonius of Perga \cite{apo}  and Pappus of Alexandria \cite{pap},
with the Hirzebruch-Jung continued fraction algorithm, 
 \cite{ewa, ful, oda},  and 
the Morelli-W\l odarczyk solution of the weak
Oda conjecture on the factorization of toric varieties.
\cite{mor,wlo},

In Theorem 
\ref{theorem:polyhedra},
Problem \eqref{equation:theproblem} is shown to be
decidable  for
 {\it rational 
polyhedra},  i.e., finite union of simplexes with
rational vertices, \cite{sta}.  In a final remark this
result is compared with Markov's  
theorem (\cite{chelek, sht},
see Theorem \ref{theorem:markov})
  on the unrecognizability of
rational polyhedra  in $\GLnQ$-geometry,
 to the effect that  manifolds cannot be characterized up to homeomorphism by computable complete invariants.

\section{Classification of rational affine spaces in $\GLnZ$-geometry}
\label{section:affine space}
A {\it rational affine hyperplane} $H$
is  a subset of $ \mathbb R^n$  of the form 
$
H =\{z \in \mathbb R^{n}   \mid  \langle p, z\rangle = \upsilon\},
   \mbox{ for some nonzero vector
     $p \in \mathbb Q^n$ and 
     $\upsilon \in \mathbb Q$.}$ 
     Here
      $\langle \mbox{-},\mbox{-}\rangle$ denotes scalar product.
     %
     %
     %
     %
     %
%
Any  intersection  
of rational affine hyperplanes in $\mathbb R^n$
is said to be   a {\it rational affine space}  in  $\mathbb R^n$.
For any  subset $X$ of  $\mathbb R^n$,
  the   {\it affine span} 
 $\,\,\aff(X)\,\,$  is defined by stipulating that 
 a point 
 $z$ belongs to   $\aff(X)$
 iff there are   $w_1,\ldots,w_k\in X$
and   $\lambda_1,\ldots,\lambda_k \in \mathbb R$
such that $\lambda_1+\dots+\lambda_k=1$
and $z=\lambda_1 w_1+\dots+\lambda_k w_k$.
(See   \cite{hand} for this terminology.)
Equivalently, we say that  $\aff(X)$ is the set of
 {\it affine combinations}  
of elements of $X$. 
A set  $\{y_1,\ldots,y_l\}\subseteq \mathbb R^n$ is said to be
{\it affinely independent}  if    none of its  elements 
 is  an affine combination of the remaining elements.
 For $\,\,0\leq d \leq n$, a {\it d-simplex}
in $\mathbb R^{n}$ is the
 convex hull  $T = \conv(v_{0},\ldots,v_{d})$ of $d+1$ affinely
independent points $v_{0},\ldots,v_{d}\in \mathbb R^n$. 
 The {\it  vertices}
$v_{0},\ldots,v_{d}$ are uniquely determined by $T$.
$T$ is said to be a {\it rational simplex}  if its vertices are
rational.  The (affine) {\it dimension}  $\dim(T)$ is equal to $d$.

The  {\it denominator}  \,\,$\den(x)$\,\, of a rational point
  $x=(x_1,\dots,x_n)\in \mathbb Q^n$
is  the least common denominator 
  of its coordinates. 
The  vector 
$$
 \widetilde{ x } = \bigl(\den(x)\cdot x_1,
 \ldots, \,\den(x)\cdot x_n,\,\,\den(x)\bigr) 
 \in \mathbb{Z}^{n+1}
$$
is said to be the  {\it homogeneous correspondent}
of   $x$.    
The integer vector 
$ \widetilde{ x }$ is {\it primitive}, \cite[p.24]{oda}, i.e., 
the greatest common divisor of its coordinates
is equal to 1.  
Every primitive integer vector $q\in \mathbb Z^{n+1}$
whose $(n+1)$th coordinate is $>0$ is the homogeneous correspondent
of a unique rational point $x\in \mathbb R^n$, called the
{\it affine correspondent} of $q.$

With the notation of  \cite[I, Definition 1.9, p.6]{ewa},
given  vectors $v_1,\ldots,v_s \in \mathbb R^{n}$ we write 
\begin{equation}
\label{equation:cone}
{\pos}[v_1 ,\ldots, v_s]=
\{x\in  \mathbb R^{n}\mid x=\rho_0 v_0+
\dots+ \rho_s v_s, 
 \,\,0\leq \rho_0,\dots,\rho_s\in \mathbb R \}.
\end{equation}
for their {\it positive hull} in $\mathbb R^{n}$. 
Let  $t=1,2,\ldots,n$. Adopting
 the terminology of  \cite[p.146]{ewa}, 
%
%
%
%
%
%
by a  {\it $t$-dimensional
rational simplicial cone} in $\mathbb R^n$ 
we understand a set $C
\subseteq \mathbb R^n$ of the form 
$
C=  \pos[ w_1 ,\ldots, w_t],
$
for linearly independent primitive integer vectors 
$w_1 ,\ldots, w_t \in \mathbb Z^n$.  
The latter are said to be the
{\it primitive generating vectors} of $C$.  They are uniquely
determined by $C$. 
By a {\it face } of $C$ we mean the positive hull
of a subset  of $\{w_1 ,\ldots, w_t\}.$
The face of  $C$  determined by the empty
set is the singleton  $\{0\}$.  This is the only
zero-dimensional cone in  $\mathbb R^n$.

\subsection*{Farey regularity}
A rational $d$-simplex  
$T=\conv(v_0,\ldots,v_d)\subseteq \mathbb R^n$
is said to be  {\it {\rm(}Farey{\rm)} regular}
(``unimodular'' in \cite{mun-dcds}) 
if the set $\{\tilde v_0 ,\ldots, \tilde v_d\}$ of 
homogeneous correspondents of its vertices 
can be extended to a   basis  of
the free abelian group $\mathbb{Z}^{n+1}$.
Equivalently, the cone  ${\pos}[\tilde v_0,\dots,\tilde v_m]
\subseteq \mathbb R^{n+1}$
is regular in the sense of \cite[Definition V 1.10, p.146]{ewa},
or ``nonsingular'' in the sense of \cite[p.29]{ful} and \cite[p.15]{oda}, or ``unimodular'' in the sense of \cite[\S 7]{hand}.

A {\it rational triangulation in $\mathbb R^n$}  is an (always finite)
simplicial complex $\Delta$ such that the vertices
of every simplex in $\Delta$ have rational
 coordinates in $\mathbb R^n$.
The point-set union of all simplexes in $\Delta$
(called the {\it support} of $\Delta$)  is the
most general possible rational polyhedron in $\mathbb R^n$,
 \cite[Chapter II]{sta}.

A   simplicial complex
is said to be  a \emph{regular
triangulation} (of its support) if 
all its simplexes are regular.
Regular triangulations are 
 affine counterparts of  regular fans of toric algebraic
geometry, \cite[p.165]{ewa},  called ``nonsingular
fans'' in  \cite[Theorem 1.10]{oda}.
({\it Warning}: the notion of a ``regular'' triangulation given in 
 \cite[p. 387]{hand} has a different meaning.)

Regular triangulations  play a fundamental role
in $\GLnZ$-geometry, as well as in the theory of  AF C*-algebras
with lattice-ordered $K_0$-group, \cite{mun-adv, mun-lincei},
(also see \cite{mun-emmone} for recent developments).

 \begin{lemma}
 [A corollary of Steinitz exchange lemma]
 \label{lemma:steinitz}
Let \/ $\conv(x_0,\dots,x_m)\subseteq 
\mathbb R^n$ be a regular $m$-simplex.
Suppose
$\conv(y_0,\dots,y_e)\subseteq \mathbb R^n$ is
a  regular $e$-simplex and  
 $\aff(y_0,\dots,y_e)=\aff(x_0,\dots,x_e)$. Then
   $e\leq m$,  and\, 
 $\conv(y_0,\dots,y_e,x_{e+1},\dots,x_m)$ is a regular
 $m$-simplex. 
 \end{lemma}
 
 \begin{proof}
 Passing to the homogeneous correspondents 
 $\tilde x_j, \tilde y_k\in \mathbb Z^{n+1}$
 of these rational points,
 by definition of regularity we have a trivial  
 variant of Steinitz exchange lemma.
 \end{proof}

 \begin{lemma}
[A corollary of  Minkowski convex body theorem]
 \label{lemma:mink}
Given    linearly  independent  
 integer vectors  $p_1,\dots,p_m \in \mathbb R^n$
let 
 \begin{equation}
\label{equation:half-par}
\mathcal P(p_1,\dots,p_m)
=\{x\in \mathbb R^{n}\mid x=\rho_1p_1+\dots+ 
 \rho_mp_m,\,\,\, \,\,0\leq \rho_1,\dots,\rho_m<1\}
 \end{equation}
 be their  half-open parallelepiped.
 Then  $\{p_1,\dots,p_m\}$ is part of a basis of the free
 abelian group $\mathbb Z^n$ iff
the  only integer point in $\mathcal P(p_1,\dots,p_m)$ 
is the origin $0\in \mathbb R^n$. 
  \end{lemma}
  
  \begin{proof}     \cite[VII, Corollary 2.6]{bar}. 
  \end{proof}

$\mathcal P(p_1,\dots,p_m)$ is called the
{\it fundamental parallelepiped of $p_1,\dots, p_m$}
in \cite[\S 7]{hand}.

\bigskip
The following result is a routine consequence of
Lemma \ref{lemma:mink}. 
 We include the proof for later use in 
 Lemma \ref{lemma:mattutino}(iii).

 \begin{lemma}
 \label{lemma:treno} Suppose
 $p\in \mathbb Z^{m+1}$, and   
 $\{b_1,\dots, b_m,q\}$ is a basis of the free abelian group
 $\mathbb Z^{m+1}$. Then
$\{b_1,\dots, b_m,p\}$ is a basis of     $\mathbb Z^{m+1}$
iff $p=\pm q +l$  for some linear combination $l$
of  $b_1,\dots,b_m$ with integer coefficients.
 \end{lemma}
 \begin{proof}
 $(\Rightarrow)$
Let  us denote by $[b_1,\dots,b_m,p]$  the unimodular matrix
with  columns   vectors  
$b_1,\dots,b_m, \,p$.
 Let $H=\mathbb Rb_1+\dots
 + \mathbb Rb_m$ be the linear span of the vectors
 $b_1,\dots,b_m$ in $\mathbb R^{m+1}.$ Let $H^\pm$  
 be the affine hyperplanes in $\mathbb R^{m+1}$
 parallel to $H$ and containing  $\pm q$ respectively.
 By way of contradiction,  suppose 
 $p$ does not coincide with
 $\pm q +l$ for any     a linear combination $l$ 
of  $b_1,\dots,b_m$ with integer coefficients.
Since $|\det[b_1,\dots,b_m,p]|= 1=$ volume of 
$\mathcal P(b_1,\dots,b_m,p)=$
volume of  $\mathcal P(b_1,\dots,b_m,q)$, 
 then  $p\in H^+$ or $p\in H^-$, 
  say  $p\in H^+.$ Then for some linear combination 
$l'$ of $b_1,\dots,b_m$ with integer coefficients,
the translated half-open parallelepiped  
$q+l'+\mathcal P(b_1,\dots,b_m)$
contains the integer point $p\not= q+l'.$ So
$\mathcal P(b_1,\dots,b_m)$ contains the nonzero
integer point  $p-q-l'.$ By Lemma \ref{lemma:mink},
this contradicts the fact that $\{b_1,\dots,b_m\}$ is part of
a basis of   $\mathbb Z^{m+1}$.
The  $(\Leftarrow)$ direction is trivial.
 \end{proof}
%
%

\subsection*{The invariants   $d_F$ and $c_F$}
\label{definition:dc}
For  every  rational affine space $F\subseteq \mathbb R^n$
 we set 
$$d_F=\min\{\den(v)\mid v\in  F\cap  \mathbb Q^n\}.$$
Next suppose  $F$  is 
 $e$-dimensional.  If   $0\leq e <n$
 we define the integer  
 $c_F>0$
   as the
 least possible denominator $\den(v)$  of
 a rational point $v\in \mathbb Q^n$
  such that there are points 
  $v_0,\ldots, v_e\in F\cap  \mathbb Q^n$
 making $\conv(v,v_0,\ldots,v_e)$ a regular
 $(e+1)$-simplex.
%
If $e=n$ we define  $c_F=1$.

\begin{lemma}
\label{lemma:cod1}
Let $F\subseteq \mathbb R^n$ be a  rational affine space. 
  If $\dim(F)\not=n-1$,  $c_F=1$. 
 If $\dim(F)=n-1$, then  $1\leq c_F\leq 
 \max(1,d_F/2)$ and $\gcd(c_F,d_F)=1$, where
 {\rm ``gcd''} denotes greatest common divisor. 

\end{lemma}
\begin{proof}
\cite[Lemma 9]{cabmun-etds}.
\end{proof}

The following theorem states that the triplet
 $(\dim(F),d_F,c_F)$ is a complete invariant for
 the rational affine space $F$ in  $\GLnZ$-geometry:

\begin{theorem}
[Rational affine spaces in $\GLnZ$-geometry] 
\label{theorem:classification}
Let \,\,$F$\, and $G$\,  be rational affine spaces
in  $\mathbb R^n$. Then\,\,
$F$\,\,and \,\,$G$\,\, 
have the same $\GLnZ$-orbit 
 iff  $(\dim(F),d_F,c_F)=(\dim(G),d_G,c_G)$.
\end{theorem}
\begin{proof}
\cite[Theorem 8]{cabmun-etds}.
\end{proof}

 \smallskip
\section{The computation of   $(\dim(F),d_F,c_F)$ for
$F$ an affine rational space}
\label{section:cd}

\begin{lemma} 
\label{lemma:affinemap}
There exists a Turing machine  $\mathcal T$ with the following property:
For any two  $(n+1)$-tuples
$V=(v_{0},\ldots,v_{n})$ and
$W=(w_{0},\ldots,w_{n})$    of rational points
in $\mathbb R^n$ with  $\den(v_{i})=\den(w_{i}), \,\,(i=0,\ldots,n)$,
$\mathcal T$   decides whether  
 both $\conv(v_{0},\ldots,v_{n})$ and
$\conv(w_{0},\ldots,w_{n})$ 
are regular
$n$-simplexes in $\mathbb R^n$  and,  if this is the case, 
computes  the uniquely determined map  
$\gamma=\phi_{VW}\in\GLnZ$ such that 
$\gamma(v_i)=w_i$ for each $i\in\{0,\ldots, n\}$.
\end{lemma}

\begin{proof}
Lemma \ref{lemma:mink} 
yields a decision procedure to check whether 
 $\conv(v_{0},\ldots,v_{n})$ is regular. 
If this is the case, the proof of   \cite[Lemma 1]{cabmun-etds}
yields an effective procedure to
compute the desired  map  $\gamma$.  
\end{proof}

\begin{lemma}
\label{lemma:ddd}  
There is a Turing machine which, given a rational  affine 
space  $F=\aff(a_1,\dots,a_m)\subseteq \mathbb R^n,
\,\,(m \mbox{ arbitrary, each }  a_i\in \mathbb Q^n)$
together with a rational  point $v_0\in F$ with 
$\den(v_0)=d_F\,,$ first 
computes the integer $e=\dim(F)$ and then outputs 
 points  $v_1\ldots,v_e\in  \mathbb Q^n \cap F$ with 
 $\den(v_1)=\dots=\den(v_e)=d_F$ such that
 $\conv(v_0,v_1,\dots,v_e)$  
is a regular $e$-simplex.  
\end{lemma}

\begin{proof}  The dimension $e$ of $F$ is immediately computed from the input rational points  $a_i$. We can easily pick 
rational points  $r_1,\dots,r_e\in F$ 
such that the set $R=\conv(v_0,r_1,\dots,r_e)$ is an
  $e$-simplex. With  the notation of \eqref{equation:cone}, let  
the set  $\mathcal R \subseteq \mathbb R^{n+1}$  be defined by 
 $$
 \mathcal R 
 = {\rm pos}[\tilde v_0,\tilde r_1,...,\tilde r_e].
 $$
%
%
%
%
%
The { desingularization procedure
\cite[VI, 8.5]{ewa}, \cite[p.48]{ful}} yields a complex  $\Phi$ of 
rational polyhedral cones  (for short, a  {\it fan})  in $\mathbb R^{n+1}$ such that
each $(e+1)$-dimensional cone   $C\in \Phi$
has the form $C={\rm pos}[\tilde q_0,\tilde q_1,\dots,\tilde q_e]$
for a suitable set 
$\{\tilde q_0,\tilde q_1,\dots,\tilde q_e\}$ of primitive integer vectors
 which is part of a basis of the free abelian group $\mathbb Z^{n+1}.$
 $\Phi$ is known as  a {\it regular (or nonsingular) fan} 
providing  a {\it desingularization} of $\mathcal R.$
$\Phi$ is computable by a Turing machine over input $v_0, r_1,\dots,r_e$.
The  rational points
 $q_0, q_1,\dots, q_e\in F$ are the vertices of a regular complex 
 $\Delta$ with support  $R$.  Thus in particular
$\Delta$ contains  a regular 
 $e$-simplex
$T_0=\conv(v_0,w_1,\ldots,w_e)$ having $v_0$ among its vertices.
By construction, the
set  $\{\widetilde{v}_0,\widetilde{w}_1,\ldots,\widetilde{w}_e\}$
is part of a basis of the free abelian group $\mathbb Z^{n+1}.$

If $\den(w_i)=d_F$ for all $i=1,\dots,e$ we are done.
Otherwise, let $j$ be the smallest index in  $ \{1,\dots,e\}$ 
such that $\den(w_j)>\den(v_0)$.
Then the integer vector $\widetilde{w}_j-\widetilde{v}_0$ is primitive, because
replacing $\widetilde{w}_j$ by $\widetilde{w}_j-\widetilde{v}_0$ in the set  
$\{\widetilde{v}_0,\widetilde{w}_1,\ldots,\widetilde{w}_e\}$
we obtain a part of a basis of $\mathbb Z^{n+1}$. 
 %
 %
 %
 So let the rational point $w_{j1}$ be  defined by 
 $\widetilde{w}_{j1}=\widetilde{w}_j-\widetilde{v}_0.$ Since both 
 $w_j$ and $v_0$ lie in $F$, then so does $w_{j1}.$  
 Further, the $e$-simplex
 $T_{j1}=\conv(v_0,w_2,\dots,w_{j1},\dots,w_e)\subseteq F$ is regular,
 and $\den(w_j)>\den(w_{j1})\geq \den(v_0).$ 
 Inductively, we have regular $e$-simplexes
 $T_{j1}, T_{j2},\ldots\subseteq F$  with 
 $$
 \den(w_j)>\den(w_{j1})>\den(w_{j2})>\dots \geq
 \den(v_0),\,\,\,
 \tilde w_{jt}=\tilde w_{jt-1}-\tilde v_0.
 $$
  After a finite number  $s=s_j\geq 0$ of steps we will  have
 $\den(w_{js})\leq \den(v_0)$, whence  
 $$\den(w_{js}) = \den(v_0) = d_F,$$
  by the assumed
 minimality property 
 of $\den(v_0)$.
 We then set 
 $
 v_j=w_{js}\,\,\, \mbox{  and  } \,\,   T_1=T_{js}\,,
 $
 and note that  $T_1\subseteq F$ is the
  regular $e$-simplex 
   obtained from $T_0$ replacing $w_j$ by the new vertex
 $v_j\in F$.  
Assuming  inductively  that  $T_{r+1}$ is obtained
in a similar way  by
  replacing a vertex of $T_r$ by a new vertex lying in $F$
  with denominator equal to $d_F$,     
the procedure will finally output the desired regular 
$e$-simplex
 $ \conv(v_0,\dots,v_e)\subseteq F$
   with $\den(v_0)=\dots=
 \den(v_e)=d_F\,.$
  The    computability of
 the map $(a_1,\dots,a_j)\mapsto(v_0,\dots,v_e)$
 is clear. 
\end{proof}

\begin{theorem}
\label{theorem:dddccc}
Let $F\subseteq \mathbb R^n$
 be an $e$-dimensional rational affine space
$(e=0,\dots,n)$. Then  there are rational points 
$v_0,\ldots,v_n\in F$ such that
\begin{itemize}
\item[(i)]  $\den(v_i)=d_F$\,\, for all \,\,$i\in\{0,\dots,e\}$;
\item[(ii)] $\den(v_i)=c_F$\,\, for all\,\, $i\in\{e+1,\dots,n\}$;
\item[(iii)] $\conv(v_0,\ldots,v_n)$ is a regular $n$-simplex.
\end{itemize}
Moreover, once  $F$ is presented as\,\, $\aff(a_0,\dots,a_e)$ for
some $a_0,\dots,a_e\in  \mathbb Q^n$,   the points 
 $v_0,\dots,v_e$ can be computed by a Turing machine.
\end{theorem}

\begin{proof}
The problem whether
$F$ contains rational points of a 
prescribed  denominator is    decidable, 
and whenever a solution exists, 
such a point can be explicitly found---e.g.,  
{via integer linear programming,  \cite[\S 7]{hand}}. Thus
 we first check whether  
$F$ contains some {\it  integer} point.
  If such $x$ exists then  $d_F$=1.  
 Otherwise we proceed inductively to
  check if $F$ contains a point with  denominator 
 2, 3, $\dots$. Since
  $F$  is a rational subspace of $\mathbb R^n$,
  this process terminates,  yielding a point
  $v_0\in F$ with the smallest possible
  denominator. Thus  $d_F=\den(v_0)$.
Now Lemma \ref{lemma:ddd} yields  a regular
$e$-simplex $\conv(v_0,\dots,v_e)\subseteq F$ satisfying 
$\den(v_1)=\dots=\den(v_e)=\den(v_0)=d_F.$
The proof now proceeds arguing by cases:

\medskip
\noindent
{\it Case 1:} $F$ has  codimension 1,
i.e.,  $e=n-1$.

Let  $C_0$  be a closed   cube with rational vertices,   
 centered at $v_0$ and containing the simplex
$\conv(v_0,\dots,v_e)$.  Let 
  $C_0\subseteq C_1\subseteq  C_2\subseteq\dots$
  be a sequence
of  closed $n$-cubes with rational vertices, 
 centered at $v_0$, where $C_{t+1}$ is obtained
by doubling the sides of $C_t$.
For any   $t=0,1,\dots, $   we
check whether there exists a rational point $s\in C_t$
satisfying the conditions
\begin{enumerate}
\item[(*)]  $ \den(s) \leq \max(1, d_F/2)$, and
\item[(**)] 
 the set  $\conv(v_0,\dots,v_e, s)$ 
is a   regular  $n$-simplex.
\end{enumerate}
 Each cube $C_t$ contains only finitely many rational points $x$
satisfying $\den(x) \leq \max(1, d_F/2).$ 
 For any such point
 $x$, Lemma \ref{lemma:mink}
 yields a method to decide  whether
  $\conv(v_0,\dots,v_e, x)$  is a regular
   $n$-simplex: one checks that 
 the half-open $(n+1)$-dimensional
 parallelepiped
\begin{equation}
\label{equation:half}
\mathcal P(\widetilde{v}_0,\dots,\widetilde{v}_e,\tilde x)
=\{\mathbb R^n \ni x=\rho_0\widetilde{v}_0+\dots+ 
 \rho_e\widetilde{v}_e + \rho_{e+1}\widetilde{x}
 \mid  0\leq \rho_0,...,\rho_{e+1}<1\}
 \end{equation}
does not contain any  nonzero integer point.
 Lemma \ref{lemma:cod1}  in combination with
  \cite[Lemma 7]{cabmun-etds}   ensures the existence
  of a point  $s\in \mathbb R^n$
  with  $\den(s)=c_F \leq \max(1,d_F/2)$, together with 
  points $v^*_0,\dots,v^*_e\in F$, all with denominator $d_F$, such
  that  
$\conv(v^*_0,\dots,v^*_e,s)$ is a regular $n$-simplex. 
Similarly, since  $\conv(v_0,\dots,v_e)\subseteq F$
 is regular,  an application of 
  Lemma \ref{lemma:steinitz}  shows that 
 \begin{equation}
 \label{equation:faticoso}
\mbox{$\conv(v_0,\dots,v_e, s)$ is a regular $n$-simplex
with   $\den(s) \leq \max(1,d_F/2)$.}
\end{equation}
We have just shown that there is
$t=1,2,\dots,$
and a  rational point $s\in C_t$ satisfying conditions
(*) and (**) above. This result is now strengthened as follows:
\begin{equation}
\label{equation:claim}
\den(s)=c_F.
\end{equation}
If  $d_F=1$  then $c_F=1$ and by Condition (*)
we are done.  

\smallskip
So assume $d_F\geq 2.$

\smallskip
If $n=1$ then  $e=0$,  so $F=\{v_0\}$
for some  $v_0\in \mathbb Q\setminus \mathbb Z$  
$d_F=\den(v_0)\geq 2$.    
Whenever   $s\in \mathbb Q$ is such that  
$\conv(v_0,s)$ is a regular
1-simplex in $\mathbb R$ and $\den(s)\leq \max(1,\den(v_0)/2)=\den(v_0)/2$,
there is no $r$ with 
$\conv(v_0,r)$   regular and $\den(r)< \den(s)$.
As a matter of fact, say without loss of generality
$s>v_0.$  Repeated applications of 
Lemma \ref{lemma:mink}
 show: 
If $v_0<r<s$ then $\mathcal P(\tilde v_0, \tilde s)$ contains
the integer point  $\tilde r$, against the regularity of
$\conv(v_0,s)$. 
If $v_0<s<r$ then $\mathcal P(\tilde v_0, \tilde r)$ contains
the integer point  $\tilde s$, against the regularity of
$\conv(v_0,r)$.
If $v_0>r$ then    $\mathcal P(\tilde v_0, \tilde s)$ contains
the integer point  $\tilde v_0-\tilde r$, against the regularity of
$\conv(v_0,r)$.  
This settles the case  $n=1.$

\smallskip
If $n > 1 $ we  can write
$
n-1=e\geq 1\mbox{ and } \den(s)\leq d_F/2.
$
Let  
$
\mathcal D =\{u\in \mathbb Q^{n}\mid  \conv(v_0,\dots, v_e,  u)
\mbox{ is a regular $n$-simplex}\}.
$
Thus  $u\in \mathcal D$ iff  $\{\tilde v_0,\dots,\tilde v_e, \tilde u\}$
is a basis of $\mathbb Z^{n+1}$. By Condition (**) and
 Lemma \ref{lemma:treno},
$u\in \mathcal D$ iff
$\tilde u= \pm \tilde s + c,$ for some linear combination $c$ of
$\tilde v_0,\dots,\tilde v_e$ with integer coefficients. Thus in particular, if $u\in \mathcal D$
the $(n+1)$th coordinate  $\tilde u_{n+1}$ of $\tilde u$ satisfies
$$
1\leq  \tilde u_{n+1}=\den(u) = \pm \tilde s_{n+1}+ kd_F=\pm\den(s)+kd_F, \mbox{ for some integer $k$.}
$$
This is so because the
$(n+1)$th coordinates of $\tilde v_0,\dots,\tilde v_e$
are all equal to $d_F.$  Since   
$\den(s)\leq d_F/2$ then $\den(u)\geq \den(s)$ for all
$u\in\mathcal D.$  
Since,  by \eqref{equation:faticoso},
${s^*}\in \mathcal D$ then  $c_F\leq \den(s)\leq \den(s^*)=c_F$,
which settles \eqref{equation:claim}, and concludes the proof of Case 1.

\medskip
\noindent
{\it Case 2: } 
The codimension of  $F$ is different
from  1. 

Then  by Lemma \ref{lemma:cod1},  $c_F=1$.
Using Lemma \ref{lemma:ddd}  
 we compute 
a regular simplex $\conv(v_0,\dots,v_e)\subseteq F$ with
$\den(v_1)=\dots=\den(v_e)=\den(v_0)=d_F.$
In case $\dim(\conv(v_0,\dots,v_e))=n$
  we are done.
  In case $\dim(\conv(v_0,\dots,v_e))\not=n$,
   knowledge that  $c_F=1$   simplifies the search
(within the increasing sequence of cubes $C_t$) of
the desired integer points  $v_{e+1},\dots,v_n$ such that
$\conv(v_0,\dots,v_n)$ is regular.  Regularity amounts to
the unimodularity of  the integer matrix
whose rows are the vectors $v_0,\dots,v_n$---a decidable problem.
Since each $C_t$ contains only finitely many integer points, an
 exhaustive search
in each $n$-cube  $C_t$ centered at $v_0$ will provide 
the desired points $v_{e+1},\dots,v_n$.

By construction,  the map
$(a_0,\dots,a_e)\mapsto(v_0,\dots,v_n)$ is computable.  
\end{proof}

 \begin{corollary}
\label{corollary:affine-subspace}
For any rational affine space $F\subseteq \mathbb R^n$ the  
invariants
$\dim(F), \,\, d_F$ and $c_F$ 
in Theorem \ref{theorem:classification}  are  computable.
Thus it is  decidable whether the affine spans of two sets of    
 points  $a_0,\ldots,a_m\in  \mathbb Q^n$ and 
\,\,$b_0,\ldots,b_l\in  \mathbb Q^n$
have the same $\GLnZ$-orbit.
Moreover, there is a Turing machine  $\mathcal M$  which,
whenever two  $e$-dimensional 
 rational affine spaces  $F,F'\subseteq \mathbb R^n$ 
have the same $\GLnZ$-orbit ($e=0,\dots,n$),
computes two $(n+1)$-tuples
$V=(v_0,\dots,v_n)$ and  $V'=(v'_0,\dots,v'_n)$ 
 of rational points
 with $\den(v_0)=\dots=\den(v_e)=\den(v'_0)=\dots
=\den(v'_e)=d_F$ and
$\den(v_{e+1})=\dots=\den(v_n)=\den(v'_{e+1})=\dots
=\den(v'_n)=c_F$ such that
$\conv(v_0,\dots,v_n)$ and $\conv(v'_0,\dots,v'_n)$ are 
 regular $n$-simplexes and  the map $\gamma=\phi_{VV'}\in \GLnZ$
of Lemma \ref{lemma:affinemap}  sends  $F$ onto $F'$.
\end{corollary}

\begin{proof}  From Theorem \ref{theorem:dddccc}  we obtain:
(i) rational points $v_0,\dots,v_e\in F$ such that
$\conv(v_0,\dots,v_e)$ is a regular $e$-simplex and 
$\den(v_0)=\dots=\den(v_e)=d_F$; 
(ii)    (if $e<n$) additional points  
$v_{e+1},\dots,v_n\in \mathbb R^n$
such that
$\conv(v_0,\dots,v_n)$ is a regular $n$-simplex 
with $\den(v_{e+1})=\dots=\den(v_n)=c_F$.
Similarly,  from 
 $(b_0,\ldots,b_k)$,  we get 
  a regular $n$-simplex
  $\conv(v'_0,\dots,v'_n)$ with
    $\conv(v'_0,\dots,v'_e)\subseteq F'$, and denominators as
    in (i)-(ii).  By 
   Theorem \ref{theorem:classification},   the identity 
    $(\dim(F),d_F,c_F)=(\dim(F'),d_{F'},c_{G'})$ can be effectively
    checked. If the identity holds, 
 Lemma  \ref{lemma:affinemap} and
 Theorem  \ref{theorem:dddccc}
yield  the  desired Turing machine $\mathcal M$.
 \end{proof}


\begin{remark} 
One might speculate that the  map 
$\gamma\in \GLnZ$ of  $F$ onto $F'$
in  Corollary \ref{corollary:affine-subspace}   is obtainable 
by solving 
a system of equations $p_1=0,\dots,p_k=0,$ where
each $p_i$ is a polynomial with integer coefficients
and the unknowns are integers representing  the terms 
the matrix $\gamma.$  As $n$ grows, so does
the degree of the system.  We  are then  
 faced with a  formidable subproblem of a 
  diophantine problem whose general 
  undecidability was proved by 
Matiyasevi\v{c} in his negative solution of Hilbert Tenth Problem,
  \cite{mat1}. 
Taking an alternative route,  the decidability of the
orbit problem for 
$F$ and $F'$ has been established 
 by constructing suitable  regular simplexes 
 in $F$ and $F'$,  using the classification
 Theorem \ref{theorem:classification}.
 In  the next sections, 
refinements of these  techniques
 will provide  computable complete
invariants for triangles and ellipses in $\GLnZ$-geometry. 
 \end{remark}

\section{Classification of 
angles in $\GLnZ$-geometry}
\label{section:angles}
As a special case of a rational affine space, a {\it rational line}
$L$ in $\mathbb R^n$ is a line containing at least two distinct rational
points.  Every rational point $v\in L$ determines
two  {\it rational half-lines}  in $L$ with a common 
origin $v$. 
A {\it rational oriented angle} is a pair $(H,K)$ of rational half-lines in 
$\mathbb R^n$ with  a common origin.
We will henceforth assume
that  ($n\geq 2$ and)  the affine spans of $H$ and $K$
in $\mathbb R^n$ 
are distinct  (for short, the angle $(H,K)$ is {\it nontrivial}).
Nontriviality is decidable by elementary linear algebra.
We  denote by $\widehat{HK}$ the convex portion
of the plane
$\aff(H\cup K)\subseteq \mathbb R^n$
 obtained by rotating $H$ to $K$
around $v$ in $\aff(H\cup K)$, with the
orientation from $H$ to $K$.  Given a rational oriented
angle $(H',K')$ in $\mathbb R^n$ we write
$
(H,K)\cong(H',K')
$
if there is $\gamma\in \GLnZ$ such that
$\gamma(H)=H'$ and $\gamma(K)=K'.$   
When this is the case we also write 
$\gamma\colon (H,K)\cong (H',K')$.
While  $\widehat{HK}=\widehat{KH}$, 
 Theorem \ref{theorem:angle} will show that
the condition $(H,K)\cong(K,H)$ generally fails.

\begin{lemma}
\label{lemma:mattutino}
For any rational half-line $H \subseteq \mathbb R^n$
 with  origin $v$, let
$H_{\rm reg}$ be the  set of rational points $h\in H$
such that the segment $\conv(v,h)$ is regular.  We then have: 

\medskip
(i) Any 
two distinct points  $h,k\in H_{\rm reg}$ have different denominators.

\medskip
(ii) $H_{\rm reg}$ contains a farthest point from $v$, denoted
$\mathsf{q}_H.$  This is also characterized as
 the point in $H_{\rm reg}$  with
the smallest possible denominator. 

\medskip
(iii) Let $(H,K)$ be a rational oriented angle in  $\mathbb R^n$ 
(with vertex $v$). 
 Let  $\mathcal L_{HK}$ be the set 
of rational points  $y\in \widehat{HK}$ such that 
$\conv(v,\mathsf{q}_H,y)$ is regular and $\den(y)$
is as small as possible. Then there exists a (necessarily unique)
point $\mathsf{p}_{HK}\in \mathcal L_{HK}$   
 nearest to $K$.
\end{lemma}

\begin{proof} It is easy to see that
$H_{\rm reg}$ is an infinite set of rational points having $v$ as
an accumulation point.

(i) By way of contradiction,
assume  $h$ and $k$ are distinct points of $H_{\rm reg}$
with   $\den(h)=\den(k)$. Passing to
homogeneous correspondents in $\mathbb R^{n+1}$
and recalling \eqref{equation:half-par},
it follows   that  either  parallelogram   
$\mathcal P(\tilde v,\tilde h)$ or $\mathcal P(\tilde v, \tilde k)$,
say   
$\mathcal P(\tilde v,\tilde h)$,
contains a nonzero integer point.
By Lemma \ref{lemma:mink},  
$\{\tilde v,\tilde h\}$   cannot be extended to a basis
of  $\mathbb Z^{n+1}$, i.e., 
the segment $\conv(v,h)$   is not regular---a contradiction.

\smallskip
(ii)  immediately follows from  (i) and  Lemma \ref{lemma:mink}. 

\smallskip
(iii) Since any two points in $\mathcal L_{HK}$  have
equal denominators, the  (infinite) set  
$\mathcal L_{HK}$  has no accumulation
points. From the proof of  Lemma
\ref{lemma:treno} 
 it follows  that
$\mathcal L_{HK}$ is contained in 
a uniquely determined  
 rational half-line
 $M \subseteq \widehat{HK}$ parallel to $H$,
 whose origin lies in  ${K}$.  
This ensures the existence and uniqueness of the  
point   $\mathsf{p}_{HK}$ nearest to $K$.
\end{proof}

\medskip

The following theorem provides a computable complete
invariant for rational oriented angles   $\GLnZ$-geometry:
 \medskip 

\begin{theorem}
[Rational  oriented angles in $\GLnZ$-geometry]
\label{theorem:angle}
For any   rational oriented angle
$(H,K)$ in $\mathbb R^n$ with vertex $v$,
 let \,\,$\mathsf{angle} (H,K)$
be the following sextuple:

\begin{itemize}
\item[(i)] The triple of integers 
$(\den(v), \den( \mathsf{q}_H)$,  $\den(\mathsf p_{HK}))$.

\smallskip
\item[(ii)]  
The (first two)   barycentric coordinates of $\mathsf{q}_K$ 
with respect to  the oriented
 triangle $\conv(v,\mathsf{q}_H,\mathsf{p}_{HK})$.

\smallskip
\item[(iii)]  
The integer  $c_{\aff(\widehat{HK})},$
(which,  by Lemma \ref{lemma:cod1},
is dispensable when  $n\not=3$).
\end{itemize}
Then the map  \,\, $(H,K)\mapsto
\mathsf{angle} (H,K)$ \,\,  is   computable.
Given a
 rational oriented angle    $(H',K')$ in $\mathbb R^n$,
 we have $(H,K)\cong(H'K')$   
\,\, iff \,\, $\mathsf{angle} (H,K)=\mathsf{angle} (H',K')$. 
Thus  the orbit  problem
\eqref{equation:theproblem}
 for rational angles  
 in $\mathbb R^n$ is decidable. 
\end{theorem} 

\begin{proof}
The definition of the rational points
 $\mathsf{q}_H,\mathsf{p}_{HK}, \mathsf{q}_K$
 in Lemma \ref{lemma:mattutino}  ensures their 
 computability.  The
  barycentric coordinates of $\mathsf{q}_K$ 
with respect to 
$\conv(v,\mathsf{q}_H,\mathsf{p}_{HK})$
are rational and computable.
 The  computability of the sextuple  $\mathsf{angle} (H,K)$
now  follows by  
 Corollary \ref{corollary:affine-subspace}.

In order to prove completeness of the invariant, 
let us suppose   $\eta\in \GLnZ$ maps  
 $H$ onto $H'$ and  $K$ onto $K'$, 
 $\eta\colon (H,K)\cong (H',K').$ 
 It follows that $\eta$ maps $\widehat{HK}$
 onto $\widehat{H'K'}.$ 
Since
  $\eta$ preserves
  regular simplexes and 
denominators of rational points, 
 Lemma \ref{lemma:mattutino} yields  the identities 
$\eta(\mathsf{q}_H)=\mathsf{q}_{H'}$,
$\eta(\mathsf{q}_K)=\mathsf{q}_{K'}$ and
  $\eta(\mathsf{p}_{HK})=\mathsf{p}_{H'K'}$.
The denominators of $v, \,\,\mathsf{q}_H,\,\, 
\mathsf{p}_{HK}, \,\,\mathsf{q}_K$ 
respectively coincide with
the denominators of $v',\,\, \mathsf{q}_{H'},\,\, 
\mathsf{p}_{H'K'}, \,\,\mathsf{q}_{K'}$.
Since   
$\eta$ preserves
affine combinations,  the  points
$\mathsf{q}_{K}$ and $\mathsf{q}_{K'}$
have the same  barycentric coordinates
with respect to the  triangles\,\, $\conv(v,\mathsf{q}_{H},
\mathsf{p}_{HK})$ \,\,and\,\,
$\conv(v',\mathsf{q}_{H'},\mathsf{p}_{H'K'})$.
Since  $\eta$ maps  $\aff(\widehat{HK})$ 
onto $\aff(\widehat{H'K'})$, 
Theorem \ref{theorem:classification}  yields
 $c_{\aff(\widehat{HK})}=c_{\aff(\widehat{H'K'})}$. 
Thus
$\mathsf{angle} (H,K)$\,\,
=\,\,$\mathsf{angle} (H', K')$.

Conversely, assume 
$\mathsf{angle} (H,K)=\mathsf{angle} (H',K')$, 
with the intent of proving  $(H,K)\cong (H',K').$
From the pair $(H,K)$ we compute
the subspace $F=\aff(\widehat{HK})$. 
 Since  $\conv(v,\mathsf{q}_H,\mathsf{p}_{HK})$ is a regular
 2-simplex, 
combining 
 Lemma  \ref{lemma:steinitz} and
 Theorem   \ref{theorem:dddccc}, 
 we obtain an $(n-2)$-tuple ${\bf a}=(a_1,\dots,a_{n-2})$ 
 of  rational points in $\mathbb R^n$, all 
with 
denominator $c_F$,  such that
$\conv(v,\mathsf{q}_H,\mathsf{p}_{HK},{\bf a})$
 is a regular $n$-simplex
in $\mathbb R^n$. Letting   $F'=\aff(\widehat{H'K'})$,
we  similarly compute
an $(n-2)$-tuple ${\bf a}'=(a'_1,\dots,a'_{n-2})$ 
of  rational points in $\mathbb R^n$, all with
denominator $c_{F'}$, in such a way that 
$\conv(v,\mathsf{q}_{H'},\mathsf{p}_{H'K'},{\bf a}')$
 is a regular $n$-simplex in $\mathbb R^n$.
Since $\mathsf{angle} (H,K)=\mathsf{angle} (H',K')$,
 the denominators of the vertices of 
 $\conv(v,\mathsf{q}_H,\mathsf{p}_{HK},{\bf a})$ 
are pairwise equal to the  denominators of the vertices of
$\conv(v,\mathsf{q}_{H'},\mathsf{p}_{H'K'},{\bf a}')$. 
Lemma \ref{lemma:affinemap} now  yields
a uniquely determined  map   $\theta\in \GLnZ$ 
such that
$\theta(\conv(v,\mathsf{q}_H,\mathsf{p}_{HK},{\bf a}))=
 \conv(v,\mathsf{q}_{H'},\mathsf{p}_{H'K'},{\bf a}').$
In more detail,   $\theta(\mathsf{q}_H)=\mathsf{q}_{H'}$
 and $\theta(v)=v'$, and hence
 $\theta(H)=H'.$  
 (Should we choose others   $(n-2)$-tuples
 ${\bf b},{\bf b}'$    
of points  with  denominator $c_{\aff(\widehat{HK})}
=c_{\aff(\widehat{H'K'})}$, such that
$\conv(v,\mathsf{q}_H,\mathsf{p}_{HK},{\bf b})$
and
$\conv(v',\mathsf{q}_{H'},\mathsf{p}_{H'K'},{\bf b}')$
are  regular $n$-simplexes, 
  the  resulting  map
 $\theta'\in \GLnZ$ given by
Lemma \ref{lemma:affinemap}
  would  still  
agree  with $\theta$ on 
 $\aff(\widehat{HK})$.)
 Since  $\theta$ preserves affine combinations
 and $\theta(\mathsf{p}_{HK})=\mathsf{p}_{H'K'}$, 
 from
 $\mathsf{angle} (H,K)=\mathsf{angle} (H',K')$
it follows that 
$\theta(\mathsf{q}_K)=\mathsf{q}_{K'}$. 
We have just proved $\theta\colon 
  (H,K) \cong  (H',K')$, as desired to complete
   the proof of the theorem. 
 \end{proof}

  \begin{remark}
 \label{remark:vertical}
{\it In  $\GLnZ$-geometry
vertical angles  need not have the same $\GLnZ$-orbit.} 
For instance, let $L$ be the  $x$-axis in $\mathbb R^2$,
and $M\subseteq \mathbb R^2$ be the   line passing through 
the points $v=(3/5,0)$ and $w=(1,1)$.
Let $L'\subseteq L$ be the half-line originating
at $v$ along the positive
direction of $L$. Let the half-line $L''$ be defined
by  $L''=\cl(L\setminus L')$, where ``cl'' denotes closure.
Let $M'\subseteq M$ be the half-line
originating at $v$ and lying in the first quadrant.    
Let $M''=\cl(M\setminus M').$  A straightforward
computation shows that 
$\mathsf{angle} (L'',M'')\not=\mathsf{angle} (L',M')
\not=\mathsf{angle} (M'',L'')$.
Therefore,  the two vertical angles
$\widehat{L'M'}$ and $\widehat{L''M''}$ 
do not have the same $\GLtwoZ$-orbit.
 \end{remark}

 While prima facie our computable complete invariant 
$\mathsf{angle} $  in 
Theorem \ref{theorem:angle}
 may look less elementary
 than its  euclidean
counterpart,\footnote{if indeed arc length or the circular
functions  are  more elementary than
our orbit-invariant $\mathsf{angle}$  in
 $\GLnZ$-geometry.}
the following proposition  shows that  any other
 computable complete $\GLnZ$-orbit  invariant
  of rational oriented angles in $\mathbb R^n$
  is Turing-equivalent to our invariant $\mathsf{angle}$. 
 

 \begin{proposition}
 [Universal property of  the invariant $\mathsf{angle} $]
 \label{proposition:equivalence-bis}
 Suppose  $\mathsf{newangle}$ is a
 computable complete  invariant
  of rational oriented angles in $\mathbb R^n$.
Then there is a Turing machine $\mathcal R$ which, over any 
 input string $\alpha=\mathsf{newangle}(H,K)$  outputs the string
  $\mathcal R(\alpha)=\mathsf{angle} (H,K)$. 
Conversely there is a Turing machine $\mathcal S$ which, over any 
 input string   $\beta=\mathsf{angle} (M,N)$,  outputs the string
 $\mathcal S(\beta)=\mathsf{newangle}(M,N)$.
 Further,
 \begin{equation}
 \label{equation:sera}
 \mbox{the two maps  $\alpha\mapsto \mathcal R(\alpha)$ 
 and $\beta \mapsto \mathcal S(\beta)$ are inverses of each other.}
 \end{equation}
 \end{proposition}
 
 \begin{proof} Suppose $\alpha =\mathsf{newangle}(H,K)$
 for some  rational angle  $(H,K)$. Equipping
 with some  lexicographic order  the set of
  all strings denoting rational angles in $\mathbb R^n$ 
  and letting
\begin{equation}
\label{equation:list-bis}
(I,J)^0,\,\, (I,J)^1,\dots
\end{equation}
  be their enumeration in this order,
  after a finite number of steps the first oriented
  rational angle  $(I,J)^t$ satisfying
  $\alpha=\mathsf{newangle}((I,J)^t)$
  will be detected.
  This follows from 
  our assumption about $\alpha$ and the
  computability
  of $\mathsf{newangle}.$
  The computability and completeness of  both  
  $\mathsf{angle} $ and $\mathsf{newangle}$  
  now yield  a Turing machine
  $\mathcal R$  
  computing, over input $\alpha$, the transformation  
  $\alpha \mapsto (I,J)^t \mapsto
   \mathsf{angle} ((I,J)^t)=
  \mathsf{angle} (H,K).$
(In case  $\alpha\not=\mathsf{newangle}((I,J)^k)$  
for all rational angles $(I,J)^k,$ \,\,\,
$\mathcal R$ will enter an infinite loop.
 {\it We are not
   assuming that the range of the invariant 
   $\mathsf{newangle}$ is decidable.}
  So in general  we cannot upgrade    $\mathcal R$  to a machine
      $\mathcal R^+$ that terminates after a finite number
      of steps  over any possible 
      input.)

  Conversely, suppose $\beta=\mathsf{angle} (M,N)$.
   The computability of
  $\mathsf{angle} $ similarly yields an effective procedure to
  detect
the first  oriented
  rational angle  $(I,J)^r$ in the list 
  \eqref{equation:list-bis}   such that
  $\beta=\mathsf{angle} ((I,J)^r)$.
  Again, the  computability and completeness  of  
  both invariants
  $\mathsf{newangle}$ and   $\mathsf{angle} $ yield  a Turing machine
  $\mathcal S$  
  computing the transformation
    $\beta \mapsto 
(I,J)^r \mapsto \mathsf{newangle}((I,J)^r)=
  \mathsf{newangle}(M,N).$   
Finally, \eqref{equation:sera} follows
from   the  completeness of  the invariants 
$\mathsf{newangle}$ and   $\mathsf{angle}.$  
 \end{proof}

 With the same proof,  
the computable complete
  invariant for
  segments,  triangles and ellipses constructed in the 
  next two sections 
  have the same
universal property.

\begin{figure} 
    \begin{center}                                     
    \includegraphics[height=8.9cm]{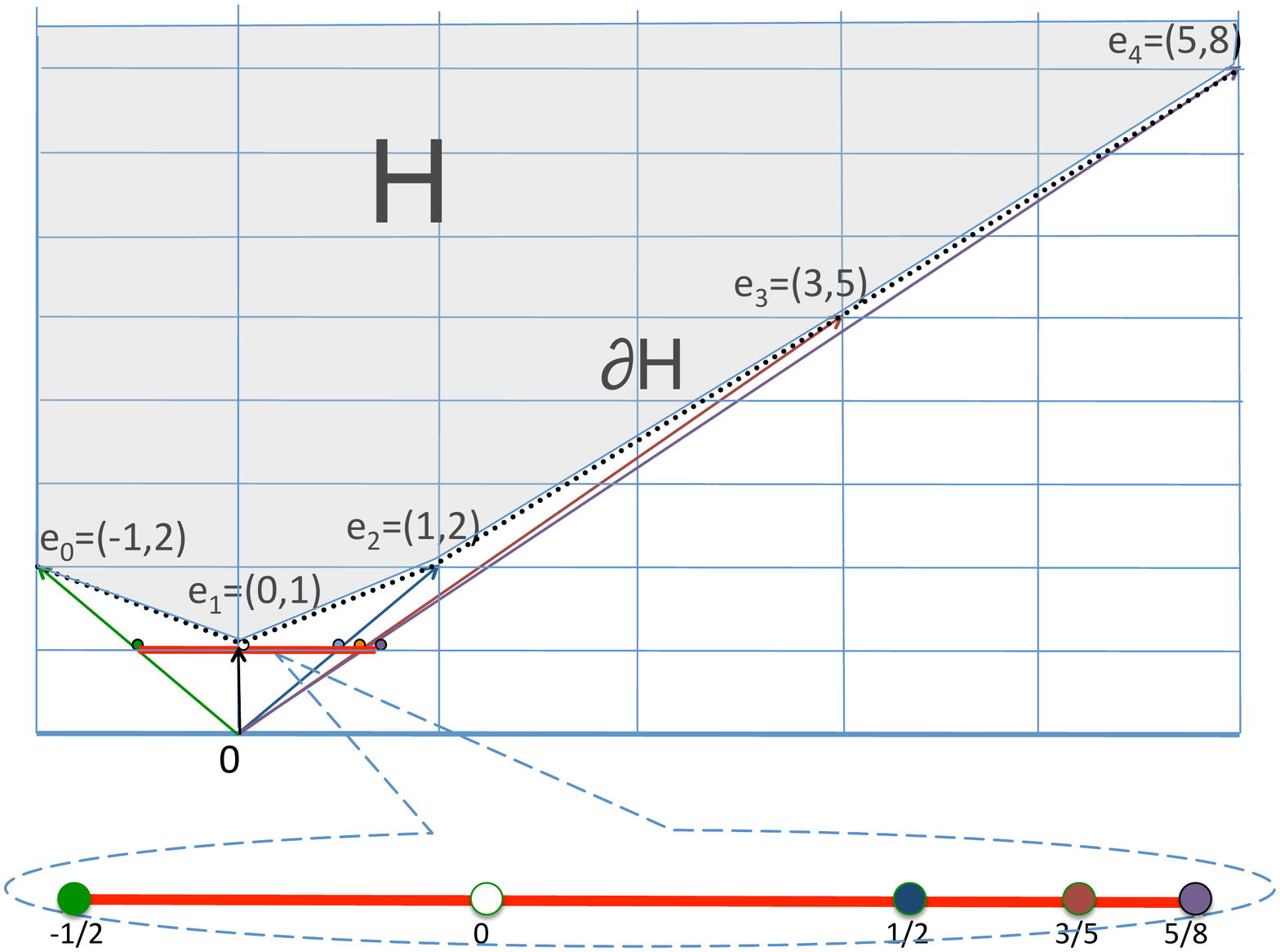}    
    \end{center}                                       
 \caption{\small  The Hirzebruch-Jung desingularization  $\Phi$ of the
 cone $C={\rm pos}[(-1,2),(5,8)]\subseteq \mathbb R^2$,
  and its corresponding regular triangulation
 $\mathsf{HJ}(A)$ of the rational segment 
 $A=\conv(-1/2,5/8)\subseteq \mathbb R$. 
 $\Phi$ is the regular fan whose primitive generating vectors
 are $e_0,\dots,e_4$.   $\mathsf H$ is the convex
 hull of the set of nonzero integer points of $C$. 
 $\partial \mathsf H$ is the relative  boundary   of $\mathsf H$.  
}  
    \label{figure:hjtris}                                                 
   \end{figure}

\section{Classification of  triangles in $\GLnZ$-geometry}
 \label{section:triangles}
\subsection*{The Hirzebruch-Jung  algorithm:
notation and terminology}
For any pair  $(a,b)$ of distinct points in $\mathbb Q^n$ let
us equip the  rational
segment  $A=\conv(a,b)\subseteq \mathbb R^n$ 
 with the orientation from $a$ to $b$.   $A$ is said to be
 an {\it oriented rational segment.}  With 
 the notation of \eqref{equation:cone}, let
  $C=\pos[\tilde a,\tilde b]\subseteq \mathbb R^{n+1}$
  be the positive hull of the homogeneous correspondents
  of $a$ and $b$.
  Let further 
\,$N$ be the set of nonzero integer points in 
$C$, and $\mathsf H = \conv(N)$, with its relative
boundary $\partial \mathsf H$.
Following
\cite[p.24-25]{oda} or  \cite[2.6]{ful}
we write  
$$\{e_0=\tilde a, e_1,   e_2,\dots,
 e_u, e_{u+1}=\tilde b \}\subseteq \mathbb Z^{n+1}
 \subseteq \mathbb R^{n+1}$$  
 for  the set of integer points  lying
on the compact edges of    $\partial \mathsf H$,
listed in the order $\sqsubseteq$  from 
$\tilde a$ to $\tilde b$.  Thus
$e_i\sqsubseteq e_j$ iff the angle  $\widehat{e_00e_i}$ 
 is contained in the
angle $\widehat{e_00e_{j}}$.
 The complex of cones  (i.e., the fan) in 
 $\mathbb R^{n+1}$  whose primitive 
 generating vectors are the  
integer vectors ${e}_i$ is said to be obtained via
 the {\it Hirzebruch-Jung (continued fraction, 
 desingularization) algorithm} on the cone $C$, \cite[p.46]{ful}.
 Let the points $x_i\in \mathbb Q^n$ be defined by 
 $$\tilde x_i=e_i,\,\,\,
(i=1,\dots,u), \,\,\,x_0=a, \,\,\,x_{u+1}=b.$$ 
 We will use the notation
 \begin{equation}
 \label{equation:hj}
  \mathsf{HJ}(A) = \mbox{ the triangulation of 
$\conv(a,b)$  with  vertices  
$x_0,x_1,\dots,x_u, x_{u+1}$}
\end{equation}
  listed in the order inherited from
$\sqsubseteq$.

\begin{proposition} 
\label{proposition:hj}
The Hirzebruch-Jung   algorithm of 
the oriented rational segment 
$A=\conv(a,b)\subseteq \mathbb R^n$
outputs a  list 
of  rational points $a=x_0,x_1,\dots,x_u, x_{u+1}=b$  
 having the following properties, for each $i=0,\dots, u$:
\begin{itemize}
\item[(i)]  The segment $\conv(x_i,x_{i+1})$  is regular.

\item[(ii)]
 $x_{i+1}$ is the 
(necessarily unique)   rational
 point  $z\in A$ with the smallest possible denominator 
 such that the segment   
$\conv(x_i,z)$ is regular. 
Equivalently, 
$x_{i+1}$ is  the
farthest point $z$ from $x_i$ in $A$ such that
$\conv(x_i,z)$ is regular. 

\item[(iii)]
 The vertices of $\mathsf{HJ}(A)$ are a subset of the vertices
 of every regular triangulation
of $A$. Thus the list
$x_0,x_1,\dots,x_u, x_{u+1}$ is uniquely determined
by $(a,b).$ 

\item[(iv)]  The map
$A\mapsto \mathsf{HJ}(A)$ is computable. 
\end{itemize}
\end{proposition}

\begin{proof}
This is the affine counterpart  of 
\cite[2.6]{ful} or  \cite[Proposition 1.19]{oda}.
The computability of the
map  
$A\mapsto \mathsf{HJ}(A)$
 is  evident  by definition  of
the Hirzebruch-Jung algorithm.
\end{proof}

The Hirzebruch-Jung regular triangulation of
 the   segment $\conv(-1/2,5/8)$  is shown  in 
  Figure \ref{figure:hjtris}.

\subsection*{Blow-up and blow-down}
Suppose $\Delta$ and $\nabla$
are two simplicial complexes in $\mathbb R^n$ with the same
support.
We say that  $\nabla$ is  
 a  {\it subdivision} of  $\Delta$
 if   every simplex of $\nabla$ is
 contained in some simplex of $\Delta$.
 Let $\Delta$ be a simplicial
 complex  and $c \in |\Delta|$.
 The {\it blow-up $\Delta_{(c)}$ of $\Delta$ at
 $c$} is the simplicial complex in $\mathbb R^n$
  obtained  
 by the following  procedure
 (\cite[p. 376]{wlo},   \cite[III, 2.1]{ewa}):
 \begin{quote}
 Replace every simplex $S \in \Delta$ such that $c\in S$ by the set
of all  simplexes of the form $\conv(c,F)$, where
 $F$ is any face of $S$ such that   $c\notin F$.
 \end{quote}
 Note that $\Delta_{(c)}$ is a subdivision of $\Delta$
  with the same support
 of $\Delta$.
 The inverse of a blow-up is called a {\it blow-down}.

For any  $m\geq 1$ and regular $m$-simplex
$T =\conv(v_{0},\ldots,v_{m})\subseteq \mathbb R^n$,  the
{\it Farey mediant of $T$}   is the affine
correspondent
of the  vector $\tilde{v}_0+\cdots+\tilde{v}_m
\in \mathbb Z^{n+1}$,  where each
$\tilde{v}_i$ is the homogeneous correspondent
of ${v}_i$.
In the particular case when $\Delta$
is a regular triangulation and
$c$ is the {\it Farey mediant} of
a  simplex  of $\Delta$,   the blow-up
$\Delta_{(c)}$ is  regular.

\subsection*{The $\GLnZ$-invariant measure $\lambda_1$}  
The second main tool for the
classification  of rational segments in  $\GLnZ$-geometry
is the one-dimensional fragment    $\lambda_1$
of the rational measure $\lambda_d$
introduced in    \cite[Theorem 2.1]{mun-dcds}.


\medskip

For any oriented rational segments 
 $\conv(a,b)$ and $\conv(a',b')$
 in $\mathbb R^n$
we  write  $\conv(a,b)\cong conv(a',b')$ if there is 
  $\gamma\in \GLnZ$ such that  $\gamma(a)=a'$ and
  $\gamma(b)=b'$, in symbols,
  $$
  \gamma\colon \conv(a,b)\cong \conv(a',b'). 
  $$

\begin{theorem}
\label{theorem:lambdaone}
   For any  rational oriented 
   segment $\conv(a,b)\subseteq \mathbb R^n$ let
\begin{equation}
\label{equation:sum}
\lambda_1(\conv(a,b))=\sum_{i=0}^u \frac{1}{\den(x_i)\den(x_{i+1})}\,\,,
\end{equation}
where $a=x_0,\, x_1,\dots,x_u, x_{u+1}=b$ are the consecutive vertices of
the Hirzebruch-Jung desingularization  
$\mathsf{HJ}(\conv(a,b))$.
We then have:

\begin{itemize}

\medskip

\item[(i)] {\rm(Invariance)} $\conv(a,b)\cong \conv(a',b')$ 
$\Rightarrow$ 
$\lambda_1(\conv(a,b))=\lambda_1(\conv(a',b'))$.

\medskip
\item[(ii)] {\rm (Computability)}  The map $(a,b)
\mapsto\lambda_1(\conv(a,b))$
  is computable.

\medskip
\item[(iii)] {\rm (Independence)}
Let $\lambda_{1}(\conv(a,b), \nabla)$ denote the result of the
 computation of
$\lambda_{1}(\conv(a,b))$   in 
\eqref{equation:sum} by means of a regular triangulation
 $\nabla$ of $\conv(a,b)$, in place of 
 $\mathsf{HJ}(\conv(a,b))$. Then
 $\lambda_1(\conv(a,b))$ = $\lambda_{1}(\conv(a,b), \nabla)$. 

\medskip
\item[(iv)] {\rm (Monotonicity)} For every rational point
$c \in\mathbb R^n$ with
  $\conv(a,b) \subsetneqq \conv(a,c)$
we have 
$\lambda_1(\conv(a,b)) < \lambda_1(\conv(a,c)).$
\end{itemize}
\end{theorem}


\begin{proof} (i)-(ii) These are immediate consequences of  
{Proposition  \ref{proposition:hj}(i)-(ii)},
 because every map  $\gamma\in \GLnZ$
preserves denominators and regularity.

\medskip

(iii)  
Let  $\Delta=\mathsf{HJ}(\conv(a,b))$.
 Let  $\Delta_{(e)}$
be  obtained 
  by blowing-up $\Delta$
  at the Farey mediant
  $e$ of some
$1$-simplex $S=\conv(x_{i},x_{i+1}) \in \Delta$.
From the regularity of $S$ we get
$\den(e) =    \den(x_{i})+\den(x_{i+1}).$
Then a routine verification shows 
 that  $\lambda_{1}(\conv(a,b), \Delta)=
\lambda_{1}(\conv(a,b), \Delta_{(e)}).$ 
The affine version of the
Morelli-W{\l}odarczyk
theorem on decomposition of
birational toric maps (solution of the weak Oda conjecture,
\cite{mor}, \cite[13.3]{wlo}), 
yields  a sequence
$$\Delta_{0}=\Delta,\Delta_{1},\ldots, \Delta_{r-1},
\Delta_{r}=\nabla$$
of regular triangulations of $\conv(a,b)$ such that
for each $t=0,\ldots,r-1$,
$\Delta_{t+1}$ is obtained from
$\Delta_{t}$ by a blow-up
at the Farey mediant  of
some simplex  of  $ \Delta_{t}$,
or vice versa,  with the roles of $t$
and $t+1$ interchanged.  (Actually, in the
present one-dimensional case we may insist that
all blow-ups precede all blow-downs.)  As  in the case
$t=0$, also for  each  $t =1,2,\ldots,r-1,$ we have the identity
$\lambda_{1}(\conv(a,b),\Delta_{t})=\lambda_{1}(\conv(a,b),\Delta_{t+1})$.

\medskip
(iv)  Let $\Delta'$
and $\Delta''$  be the Hirzebruch-Jung 
desingularizations of
$\conv(a,b)$ and $(\conv(b,c)$  respectively. 
  Then $\Delta'\cup \Delta''$ determines a regular 
  complex with support $\conv(a,c).$ By
\eqref{equation:sum} and (iii),
$$\lambda_1(\conv(a,b)) < \lambda_1(\conv(a,c),\Delta'\cup\Delta'')
=\lambda_1(\conv(a,c)).$$ 
The proof is complete. 
\end{proof}

\medskip

\begin{theorem}
[Rational oriented segments in  $\GLnZ$-geometry]
\label{theorem:quadruple}
Let 
$A=\conv(a,b)$
be a rational oriented segment in $\mathbb R^n$.
Let  $x_1\not=a$ be the point nearest to $a$ in the
Hirzebruch-Jung triangulation of $A$.
Then the quadruple 
$$\mathsf{side} (A)=\{c_{\aff(A)},\,\,\lambda_1(A),\,\,
\den(a),\,\,\den(x_1)\}$$
is  a  
 computable complete  
  $\GLnZ$-orbit invariant for $A$.
If  $n\not=2$,   the integer
 $c_{\aff(A)}$  is redundant.
\end{theorem}

\begin{proof} 
For any rational oriented segment $S$ in $\mathbb R^n$
let us write 
$$\mathsf{hj}(S)=(\den(\mathsf{HJ}(S)), c_{\aff(S)}),$$
 where
$\den(\mathsf{HJ}(S))=
(\den(s_0), \den(s_1),\dots,\den(s_q),\den(s_{q+1}))
$
is  the sequence 
of the denominators of the vertices of
the Hirzebruch-Jung triangulation  $\mathsf{HJ}(S)$,
  listed in the $\sqsubseteq$-order. 
By Corollary \ref{corollary:affine-subspace}, 
the  integer
$c_{\aff(S)}$ is computable.
{By Proposition \ref{proposition:hj}(iv),} the map
$S\mapsto $ $\mathsf{hj}(S)$ 
is  computable.

\medskip
\noindent
{\it Claim 1.} 
For any  oriented rational segment
$A'=\conv(a',b')$
 in $\mathbb R^n$ we have  $\conv(a,b)\cong \conv(a',b')$  iff 
 $\mathsf{hj}(A)=\mathsf{hj}(A').
$

\smallskip
$(\Rightarrow)$
Suppose  $\eta\colon A\cong A'$
for some $\eta\in \GLnZ$.
 By Theorem \ref{theorem:classification},  
 $c_A=c_{A'}$.
{By Proposition \ref{proposition:hj}(i)-(ii), } 
the two  sequences of denominators in $\mathsf{HJ}(A)$ and
$\mathsf{HJ}(A')$ coincide,
 because $\eta$ preserves denominators
and regularity. 
So $\mathsf{hj}(A)=\mathsf{hj}(A').$

\smallskip
$(\Leftarrow)$
Conversely, suppose 
$\mathsf{hj}(A)=\mathsf{hj}(A')$.
Let \,\,\,
$a= x_0,\,x_1,\dots,x_u,x_{u+1}=b$\,\,\, be the vertices of \,\,\,$\mathsf{HJ}(A),$
\, and \,\,\,
$a'= x'_0,\,x'_1,\dots,x'_u,x'_{u+1}=b'$ be the vertices of $\mathsf{HJ}(A')$.
By hypothesis,  $\den(x_i)=\den(x_i')$
for each $i=0,\dots, u+1$. Further,
 $c_{\aff(A)}=c_{\aff(A')}$.
Since $\conv(x_0,x_1)$ is regular, combining
Lemma  \ref{lemma:steinitz}  with
  Theorem   \ref{theorem:dddccc},
we have   rational points
 $w_2,\dots,w_n$, all with the same  
 denominator  $c_{\aff(A)}$, such  that 
 $\conv(x_0,x_1,w_2,\dots,w_n)$
is a regular $n$-simplex in $\mathbb R^n.$
%
Symmetrically, 
Lemma  \ref{lemma:steinitz}  and 
  Theorem   \ref{theorem:dddccc}
  yield a regular
$n$-simplex $\conv(x'_0,x'_1,w'_2,\dots,w'_n)$ in $\mathbb R^n$,
where $\den(w'_2)=\dots=\den(w'_n)= c_{\aff(A')}$.
It follows that the   denominators of the vertices of
$\conv(x'_0,x'_1,w'_2,\dots,w'_n)$ and
of $\conv(x_0,x_1,w_2,\dots,w_n)$  
are pairwise equal.
(Note that  $d_{\aff(A)}=d_{\aff(A')}$ follows from our standing
hypothesis $\mathsf{hj}(A)=\mathsf{hj}(A')$.)
Iterating this construction
for each  $t=0,\dots,u,$ we obtain
regular $n$-simplexes
$\conv(x'_t,x'_{t+1},w'_2,\dots,w'_n)$
and $\conv(x_t,x_{t+1},w_2,\dots,w_n)$
such that the denominators of their vertices  
are pairwise equal.  Thus 
  Lemma  \ref{lemma:affinemap}
yields    $\gamma_t\in \GLnZ$ 
satisfying 
   $\gamma\colon \conv(x_t,x_{t+1},w_2,\dots,w_n)
   \cong \conv(x'_t,x'_{t+1},w'_2,\dots,w'_n)$.  The map
   $t\mapsto \gamma_t$ is computable.
    
We also have 
\begin{equation}
\label{equation:gammas}
\gamma_0=\gamma_1=\dots=\gamma_u.
\end{equation}
Indeed, let  us compare 
$\gamma_0$ and $\gamma_1$. On the one hand,
  $\gamma_0$ maps     
  $\conv(x_0,x_1)$  onto $\conv(x'_0,x'_1)$, and
  $\gamma_1$ maps 
  $\conv(x_1,x_2)$ onto  $\conv(x'_1,x'_2)$.  
  On the other hand, 
  $\conv(x_1,x_2)$   is mapped onto
  $\conv(x'_1,\gamma_0(x_2))$ by 
$\gamma_0$, and is mapped onto $\conv(x'_1,x'_2)$ 
by  $\gamma_1$.   
Both segments  $\conv(x'_1,\gamma_0(x_2))$ and
$\conv(x'_1,x'_2)$ 
are regular.  Further,  $\den(x'_2)=\den(\gamma_1(x_2))=\den(\gamma_0(x_2))$,
because  $\gamma_0$ and $\gamma_1$ preserve denominators
and regularity.
By Lemma  \ref{lemma:mattutino}(i) and {Proposition
\ref{proposition:hj}(i)-(ii)}, 
$\gamma_1(x_2)=x'_2=\gamma_0(x_2).$
%
 Thus
$\gamma_0$ and $\gamma_1$ agree over
 the  segment $\conv(x_1,x_2)$,
whence they agree over all of $\mathbb R^n$, because both
$\gamma_0$ and $\gamma_1$  
send $w_2,\dots,w_n$  to $w'_2,\dots,w'_n$.
Inductively,  $\gamma_{j-1}$ agrees with $\gamma_{j}$
 over the segment $\conv(x_j,x_{j+1})$ whence
 $\gamma_j=\gamma_{j-1}$.
 Thus every $\gamma_j$ agrees with $\gamma_0$, 
 as required to settle \eqref{equation:gammas} and 
Claim 1.

\medskip
\noindent
{\it Claim 2.}  
Let $A=\conv(a,b)$ and $A'=\conv(a',b')$
be  oriented rational segments
 in $\mathbb R^n$. Then $\conv(a,b)\cong \conv(a',b')$  iff 
$\mathsf{side}(A)=\mathsf{side}(A').$

 \smallskip
  Let  
$
(x_0=a,\,\,x_1,\dots,x_l,\,\,x_{l+1}=b)
\,\,\,\mbox{ and }\,\,\,
(x'_0=a',\,\,x'_1,\dots,x'_m,\,\,x'_{m+1}=b')
$
be the lists of vertices of the Hirzebruch-Jung
desingularizations of $A$ and $A'$.
Suppose  $\eta\in \GLnZ$ maps  $A$ onto $A'$.
Then $c_{\aff(A)} = c_{\aff(A')}$.
By  Claim 1($\Rightarrow$),
$\mathsf{hj}(A)=\mathsf{hj}(A')$. 
Then \eqref{equation:sum}  yields 
  $\lambda_1(A)=
\lambda_1(A')$.  {By  Proposition
\ref{proposition:hj}(i),}
both segments
$\conv(x_0,x_1)$ and $\conv(x_0',x_1')$ are regular.    
{By Proposition
\ref{proposition:hj}(ii)},   $\eta(x_1)=x'_1$.
It follows that
$\mathsf{side} (A)=\mathsf{side} (A')$. 
 
\smallskip 
Conversely, assume
$\mathsf{side} (A)=\mathsf{side} (A')$.
%
We will prove
\begin{equation}
\label{equation:gammab}
l=m \mbox{ and for some 
$\gamma\in \GLnZ$, } \gamma(x_i)=x_i' \mbox{ for all }i=1,2,\dots, l+1.
\end{equation}
By way of contradiction, assume  $l<m$.
Since by hypothesis $\mathsf{side} (A)=\mathsf{side} (A')$, then
 $\mathsf{hj}(\conv(a,x_1))=\mathsf{hj}(\conv(a',x_1')).$ 
The  basis in the inductive construction in
Claim 1$(\Leftarrow)$ now yields  a map 
$\gamma\in \GLnZ$ of $\conv(a,x_1)$ 
 onto $\conv(a',x_1').$
{By Proposition
\ref{proposition:hj}(i)-(ii)}, $\gamma(x_2)=x'_2.$
Inductively,  $\gamma(x_i)=x'_i$
for each $i=1,\dots,l+1.$ Thus $\gamma(A)$ is a proper subset
of  $A'$.
{By Theorem 
\ref{theorem:lambdaone} (iv)},  
$$
  \lambda_1(\conv(a', \gamma(b)))
<  \lambda_1(\conv(a',b')).
$$
On the other hand, since  $\mathsf{side} (A)=\mathsf{side} (A')$, 
{from Theorem  
\ref{theorem:lambdaone} (i) we get}
$$
\lambda_1(\conv(a',b'))=\lambda_1(\conv(a,b))
= \lambda_1(\conv(\gamma(a), \gamma(b)))=
  \lambda_1(\conv(a', \gamma(b))),
$$
a contradiction that settles \eqref{equation:gammab}.
In case  $l>m$, arguing by contradiction
one similarly  proves  \eqref{equation:gammab}. 
Thus $\gamma\colon A\cong A'.$ 
 The proof of Claim 2 is complete.

 \smallskip
 By Theorem \ref{theorem:lambdaone}(ii),
the  rational   $\lambda_1(A)$ is computable.
Therefore, the map $A\mapsto \mathsf{side}(A) $ is computable.
The redundancy of $c_{\aff(A)}$ 
for  all $n\not=2$, follows from
 Lemma \ref{lemma:cod1}.
\end{proof}

Let $\conv(u,v,w)$ be  a rational 2-simplex
in $\mathbb R^n$,  with the orientation
$u\to v\to w$. We say that
 $\conv(u,v,w)$ is an {\it oriented rational triangle.} 
For any oriented rational triangle  $\conv(u',v',w')$
 in $\mathbb R^n$
we  write  $\conv(u,v,w)\cong \conv(u',v',w')$ if  
there is 
  $\gamma\in \GLnZ$ with $\gamma(u)=u'$,
  $\gamma(v)=v'$ and $\gamma(w)=w'$,
   in symbols,
  $$
  \gamma\colon \conv(u,v,w)\cong \conv(u',v',w'). 
  $$
The  rational segments $\conv(v,u)$ and $\conv(v,w)$
determine   two rational half-lines  
$H_{vu} \subseteq  \aff(\conv(v,u))$ 
 and $K_{vw} \subseteq  \aff(\conv(v,w))$ with 
 their common vertex  $v$.
 We then have   the  (nontrivial )  angle
 $ (H_{vu},K_{vw}) \subseteq \mathbb R^n$.

The  following theorem is a 
counterpart  for  $\GLnZ$-geometry
of the
side-angle-side  criterion for congruent triangles
in euclidean geometry:

\begin{theorem}
[Rational oriented triangles in $\GLnZ$-geometry]
\label{theorem:triangle} 
 A   computable  complete  invariant 
of any rational oriented triangle  
$T=\conv(u,v,w)\subseteq  \mathbb R^n$ in
$\GLnZ$-geometry
is given by 
$$\mathsf{tri} (T) =
(\mathsf{side} (\conv(v,u)),\,\,\,
\mathsf{angle} (H_{vu},K_{vw}) ,\,\,\,
\mathsf{side} (\conv(v,w))).
$$
\end{theorem}

\begin{proof} Let $T'=\conv(u',v',w')$ be another rational
oriented  triangle in $\mathbb R^n.$

If  $\gamma\colon T\cong T'$ for some $\gamma\in \GLnZ$,
then by Theorems
 \ref{theorem:angle} and   \ref{theorem:quadruple},  
 $\mathsf{tri} (T)=\mathsf{tri} (T')$.

 Conversely, suppose 
  $\mathsf{tri} (T)=\mathsf{tri} (T')$.
  Let us write for short $H=H_{vu}, \,\,
  K=K_{vw},\,\, H'=H_{v'u'}, \,\,K'=K_{v'w'}.
  $
Mimicking the proof of Theorem
 \ref{theorem:angle},
we  preliminarily compute the integer  $c_{\aff(T)}$, as well as 
  the  points $\mathsf q_{H}\in H$, $\mathsf q_{K}\in K$
and   $\mathsf p_{HK}\in \aff(T)$, and  the
regular triangle  
 $R=\conv(\mathsf q_{H},v,\mathsf p_{HK})$.
 Combining  Lemma
 \ref{lemma:steinitz} and
 Theorem  \ref{theorem:dddccc},
 we extend  $R$
to a regular $n$-simplex  $R_*=\conv(\mathsf q_{H},v,
\mathsf p_{HK}, z_3,\dots,z_n)$
such that $\den(z_3)=\dots=\den(z_n)=c_{\aff(T)}.$
We similarly extend
the regular triangle  $R'=\conv(\mathsf q_{H'},v',\mathsf p_{H'K'})$
to a regular $n$-simplex  $R'_*=\conv(\mathsf q_{H'},v',\mathsf p_{H'K'}, z'_3,\dots,z'_n)$
such that $\den(z'_3)=\dots=\den(z'_n)=c_{\aff(T')}.$
Since $\mathsf{angle} (H,K)=
\mathsf{angle} (H',K')$,  then 
$c_{\aff(T)}=c_{\aff(T')}$, and both 
$R_*$ and $R'_*$  are effectively computable.
Since by hypothesis, $\mathsf{tri} (T)=\mathsf{tri} (T')$, 
 the denominators of the vertices of  
$R_*$  and $R'_*$ are pairwise equal.
Lemma \ref{lemma:affinemap}  yields  a uniquely
determined   map $\gamma\in \GLnZ$ of 
$R_*$ onto $R'_*.$
It follows that 
$\gamma\colon R \cong R', \mbox{ and hence }
\gamma\colon H \cong H'.$
{Using Proposition \ref{proposition:hj}(iv)}  we compute the
  two $(l+2)$-tuples
  %
  %
  %
  %
 %
 %
 %
 %
$$
(x_0=v,\,\,x_1,\dots,x_l,\,\,x_{l+1}=u)\,\,\, \mbox{ and }\,\,\,
(x'_0=v',\,\,x'_1,\dots,x'_m,\,\,x'_{m+1}=u')
$$
listing  the  vertices of the Hirzebruch-Jung
desingularizations  $\mathsf{HJ}(\conv(v,u))$ 
and $\mathsf{HJ}(\conv(v',u'))$.
{From the hypothesis 
$\mathsf{side} (\conv(v,u))=\mathsf{side} (\conv(v',u'))$}
it follows that $\mathsf{hj}(\conv(v,u))=\mathsf{hj}(A\conv(v',u')).$ 
Therefore,  $l=m$ and for each $i=0,\dots,l+1$,
$\gamma$ sends the 
$i$th vertex of  $\mathsf{HJ}(\conv(v,u))$
to  the $i$th vertex of $\mathsf{HJ}(\conv(v',u'))$.
Thus  $\gamma\colon \conv(v,u)\cong \conv(v',u').$   
The assumption  $\mathsf{angle} (H,K)$ =
$\mathsf{angle} (H',K')$ entails
$
\gamma\colon K\cong K',
$
whence
$\gamma\colon\widehat{HK}\cong\widehat{H'K'}.$ 
From 
  $\mathsf{side} (\conv(v,w)) = \mathsf{side} (\conv(v',w'))$ 
it follows that 
  $\gamma\colon \conv(v,w)\cong \conv(v',w').$
Summing up, 
  $\gamma\colon T\cong T'$. 

The computability of the map  \,\, $T\mapsto
\mathsf{tri}(T)$\,\, follows from
 Theorems   \ref{theorem:angle} and
 \ref{theorem:quadruple}.
\end{proof}

\section{Classification of  ellipses   in
$\GLtwoZ$-geometry}
\label{section:ellipses}

For notational  simplicity, all ellipses in this paper are assumed
to  lie in $\mathbb R^2$.
 
Our construction
of   a computable complete
invariant for 
ellipses in $\GLtwoZ$-geometry primarily rests on the 
following  properties of conjugate diameters,   recorded
by
Apollonius of Perga \cite{apo}
 and Pappus of Alexandria \cite{pap}:

 \begin{proposition}
 \label{proposition:ellipses}
  
Every ellipse  $E\subseteq \mathbb R^2$  is the image of a circle  
under a contraction with respect to some line.    
 $E$ has a unique center  of symmetry.
Calling a {\em diameter} of $E$ any chord passing through the center, 
it follows that the center of $E$ bisects any diameter. 
Let $C$ be a diameter of $E$. There is a uniquely determined
diameter  $C^*$ having the property that the middle
points of all chords of $E$ parallel to $C$ lie in $C^*$.
The latter is known as  {\em the conjugate diameter of $C$}.
We have  $C^{**}=C$.
Let  $T$ be the tangent of $E$ at a point $x\in E$.
Let $X$ be the diameter of $E$ containing $x$. 
Then the conjugate diameter 
$X^*$ is parallel to $T$. 
 \end{proposition}
 
 \begin{proof} All these properties  
 follow from  the fact that every
 affine  transformation can be represented as a composition of a similarity
transformation and a contraction with respect to some line.
 \end{proof}
 
 \medskip
\begin{lemma}
\label{lemma:newstart}
Let $\phi(x,y)= ax^2 + bxy + cy^2 + dx + ey + f $
be a polynomial with rational coefficients.
Then it is decidable whether the solution set
(the {\em zeroset}) 
$$
Z(\phi)= \{(x,y)\in\mathbb R^2\mid \phi(x,y)=0\}
$$
is an ellipse $E$ containing a rational point. 
Further,
whenever any such point exists in $Z(\phi)$,
the set of rational points in $E$ is dense in $E$,
and can be recursively enumerated  in the
lexicographic
order of increasing denominators.
\end{lemma}

\begin{proof}
As explained, e.g.,  in \cite[\S 5.2]{computer},
$E$ contains a rational point iff
 the Legendre
 equation
$
 px^2+qy^2+rz^2=0
$
 has an integer solution   with $\gcd(x,y,z)=1$, 
for suitable integers $p,q,r$ which  are effectively
 computable from the coefficients $a,\dots,f$.
Perusal of   \cite[17.3]{ireros} shows 
that this latter problem
 is decidable, and  whenever a solution exists it can be
 effectively computed. The rest is clear.
 (See  \cite{crerus}
 and \cite{sim} for  efficient
 computations of rational points on rational conics.)
 \end{proof}

By a {\it rational ellipse} we mean an ellipse 
$E\subseteq \mathbb R^2$ 
that coincides with  the zeroset  $Z(\phi)$
of a quadratic polynomial $\phi(x,y)$ with rational 
coefficients, and contains a rational point (equivalently, $E$
contains a dense set of rational points).
We denote by $\mathcal E$ the set of rational
ellipses. 
For any  $E\in \mathcal E$, its
 center  is
a rational point.   A diameter 
$C$ of $E$ is {\it rational}  (meaning that 
its vertices are rational)
iff so is its conjugate.
Given a rational diameter $C$ in $E$, its conjugate
 is effectively computable.
Two semi-diameters $A,B$  of $E$ are said to be {\it conjugate}
iff they  lie in  conjugate diameters of $E$.

For any map $\gamma\in \GLtwoZ$, the
 image  $E'=\gamma(E)$ of any
$E\in  \mathcal E $ is a member of $\mathcal E.$
Further,
 $\gamma$ is a denominator preserving
 one-one 
 map of all rational points
 of  $E$  onto all rational points of 
 $E'$.

 \smallskip
The proof of the following  result  now routinely follows from 
 Proposition  \ref{proposition:ellipses}:

\begin{lemma}
\label{lemma:semi-diameter}
(i)  For any pair of distinct segments $C,D$ with
a common vertex and $\aff(C)\not=\aff(D)$,   there is a unique ellipse 
$E$ such that $(C,D)$ is a pair of conjugate 
semi-diameters of $E$.\footnote{Pappus  \cite[Book VIII, \S XVII, 
 Proposition 14]{pap} constructs
 the axes of an ellipse from any given pair of conjugate 
 semi-diameters.}
Further, if
the segments $C$ and $D$ are rational
then  $E$ is a rational ellipse,
which  can be effectively obtained  from
$(C,D)$.

\smallskip
(ii) 
Two  ellipses $E,E'\in \mathcal E$   
have the same $\GLtwoZ$-orbit 
 iff there are conjugate rational semi-diameters
$A,B$ of $E$ and $A', B'$ of $E'$ such that
the triangles $T=\conv(A\cup B)$ and $T'=\conv(A'\cup B')$ 
have the same $\GLtwoZ$-orbit. Moreover, if $\delta\in \GLtwoZ$
maps $T$ onto $T'$, then $\delta$ maps $E$ onto $E'$. 
\end{lemma}

\smallskip
Let $O$ be the center of $E\in \mathcal E$.
Let $(A,B)$ be a pair of
conjugate rational semi-diameters of
$E$, say $A=\conv(O,x)$ and $B=\conv(O,y)$.
Then the
sum of the denominators
of $x$ and $y$
is said to be the {\it index of $(A,B)$}.

\begin{theorem}
[Rational ellipses in $\GLtwoZ$-geometry]
\label{theorem:ellipses}
For any   $E\in \mathcal E$
let   $\{D_1,\dots,D_q\}$ be the
 set of all   pairs
$D_i=(A_i,B_i)$ of conjugate
semi-diameters  of $E$ having the smallest index.
For each $i=1,\dots,q$ let
the  triangle $T_i=\conv(A_i \cup B_i)$ be oriented
  so that $O$ is the first vertex, followed by
the vertex of $A_i$, followed by the
vertex of $B_i$.
With \,\,
 $\mathsf{tri} (T_i)$ \,\, the invariant defined in   
Theorem   \ref{theorem:triangle},  let 
$$\mathsf{ell} (E)=\{\mathsf{tri} (T_1),
\dots,\mathsf{tri} (T_q)\}.
$$
 We then have:
\begin{itemize}
\item[(i)]
$\mathsf{ell} $ is a 
complete $\GLtwoZ$-orbit  invariant
of   ellipses in $\mathcal E$.

\medskip
\item[(ii)]  For any  rational
quadratic polynomial $\phi(x,y)$  
whose zeroset $Z(\phi)$  is an
element of $\mathcal E$,
 (a decidable condition, by Lemma \ref{lemma:newstart}),
the  map $\phi \mapsto \mathsf{ell} (Z(\phi))$ is computable.

\medskip
 \item[(iii)] Thus there is a decision procedure
 for  the problem whether
 two rational ellipses $E,E'\in \mathcal E$  
have the same $\GLtwoZ$-orbit.  
 When this is the case, a map
  $\gamma\in \GLtwoZ$
of  $E$ onto $E'$ can be effectively computed.
 \end{itemize}
\end{theorem}

\begin{proof} 
(i)  For any  $E,E'\in \mathcal E$ we must show:
$$
\mbox{$E$ and  $E'$
have the same $\GLtwoZ$-orbit\,\,\,\, iff \, $\mathsf{ell} (E)
=\mathsf{ell} (E')$.}
$$

\noindent
$(\Leftarrow)$  Suppose  $\mathsf{ell} (E)=\mathsf{ell} (E')$. 
By assumption, 
  $E$ has a pair  $(A,B)$ of rational conjugate
semi-diameters
of smallest index $d$,\,\, and 
$E'$ has a pair $(A',B')$
of rational conjugate
semi-diameters of the same smallest index $d$,  such  that
the two triangles   $\conv(A\cup B)$ and $\conv(A'\cup B')$ 
have the same  invariants.
 By  Theorem  \ref{theorem:triangle},
 the two triangles $\conv(A\cup B)$
 and  $\conv(A'\cup B')$ 
 have the same $\GLtwoZ$-orbit.
By Lemma \ref{lemma:semi-diameter},
$E$ and $E'$ have the same $\GLtwoZ$-orbit.

 \medskip
\noindent
$(\Rightarrow)$ 
Let $\gamma\in \GLtwoZ$ map $E$ onto $E'$.
Let $O$ be the center of $E$, and $O'$ the center of $E'$.
Since  $\gamma$ preserves
ratios of collinear segment lengths, as well as
parallel and tangent lines, 
then by Lemma \ref{lemma:semi-diameter}, 
$O'=\gamma(O)$. Further, 
 $\gamma$ sends any pair $(A,B)$ of conjugate semi-diameters
of $E$ to a pair $(A',B')$  of conjugate semi-diameters 
of $E'=\gamma(E)$.
The preservation properties of the affine transformation
$\gamma$ ensure that  the image $\gamma(E)$ coincides
with the ellipse constructed from  $(A',B')$
according to Lemma \ref{lemma:semi-diameter}. 
Pick a triangle  $T$ of $E$ arising from a pair of semi-diameters
of smallest index $d$.   
Since $\gamma$ preserves 
all numerical invariants in 
Theorems \ref{theorem:angle}
and  \ref{theorem:quadruple}, then
   the two sides of
$\gamma(T)$  having $O'$ as a common vertex
will be conjugate semi-diameters of $E'$ of smallest
index  $=d$.  Further,  
$\mathsf{tri} (T)=\mathsf{tri} (\gamma(T))$.  
It follows  that
 $\gamma$ induces a bijection  $\beta$ between 
pairs $(A_j,B_j),  \,\,\, j=1,\dots,q,$  
of conjugate semi-diameters of $E$ of smallest index, 
and pairs $(\beta(A_j),\beta(B_j))$ 
of conjugate semi-diameters of $E'$ of smallest
index, and we may write 
$\gamma\colon \conv(A_j\cup B_j)
\cong \conv(\beta(A_j)\cup \beta(B_j))$.  
By  Theorem \ref{theorem:triangle}, 
$\mathsf{tri} (T_j)=\mathsf{tri} (\gamma(T_j))$ 
for each $j=1,\dots, q$.
By definition,   $\mathsf{ell} (E)=\mathsf{ell} (E')$,
 which completes the proof of (i).

\medskip
(ii)  To prove the computability of the map
 $\phi\mapsto \mathsf{ell} (Z(\phi))$
we preliminarily check that  the zeroset 
 $Z(\phi)$ is a member of $\mathcal E$. 
By Lemma \ref{lemma:newstart},
 this condition can be decided
  by a Turing machine over the input given by the coefficients
of $\phi$.
If the condition is satisfied, 
 letting $E=Z(\phi)$  we proceed as follows:
 
 \medskip
\begin{enumerate}

\item[] We compute the (automatically rational) center $O$ of $E$,
and let  $S$ be a closed square 
with rational vertices in $\mathbb R^2$,  centered at $O$ and containing 
$E$;

\medskip
\item[] 
 For each $j=1,2,\dots$ we let  $X_j$ be the set of rational points of
$E$ of denominator $\leq j$.    $X_j$ is effectively
computable, because there are only finitely many rational 
points  $x\in S$ of denominator  $\leq j,$ and it is decidable
whether  any such point  $x$ lies in $E$;

\medskip
\item[] 
 For any pair  $(x,y)$ of points in $X_j$ we check whether
   $(\conv(O,x),\conv(O,y))$ is a pair of conjugate semi-diameters
   of $E$. As already noted, this can be done in an effective way; 

\medskip
\item[] 
Let  $d$ be the smallest integer such that  $X_d$
contains two points  $x,y$ having the property that 
$(\conv(O,x),\conv(O,y))$  is a pair
 of  rational conjugate semi-diameters  of $E$ of index $d$.
Since   $E$ does have rational conjugate 
semi-diameters, after a finite number of steps
such $d$ will be found;

\medskip
\item[]  Let  $\{D_1,\dots,D_q\}$ be the  (necessarily finite)  set of all
pairs of conjugate semi-diameters of $E$  of index $d$.
For any $D_i=(A_i,B_i)$, letting $T_i=\conv(A_i\cup B_i)$,
we compute  
$\mathsf{tri} (T_i)$  as in   Theorem
\ref{theorem:triangle}. 
 We finally write  $\mathsf{ell} (E)=\{\mathsf{tri}
  (T_1),\dots,\mathsf{tri} (T_q)\}$.
\end{enumerate}

\medskip
\noindent
Since all these steps are effective, 
the map $\phi\mapsto \mathsf{ell} (Z(\phi))$
 is computable.

\medskip
(iii) This immediately follows from  the proof of (i) and (ii). 
The  proof of (ii)
also shows that there is
 a Turing machine having the following
property:  whenever  $E$ and  $E'$
have the same $\GLtwoZ$-orbit,
a map
 $\gamma\in \GLtwoZ$
 of  $E$ onto $E'$
 is effectively obtainable from the input data
$\phi$ and $\phi'$.
\end{proof}

The computable  complete  invariant  \,\,$\mathsf{ell} $\,\,    of the
foregoing theorem   is here to stay,
 because of the following Turing equivalence
    result, whose proof is similar to the proof of
 Proposition \ref{proposition:equivalence-bis}:

 \begin{proposition}
 [Universal property of  $\mathsf{ell} $]
 \label{proposition:equivalence}
 Suppose  $\mathsf{newell}$ is a
 computable complete  invariant
 of ellipses in  $\mathcal E$ in $\GLtwoZ$-geometry.
Then there is a Turing machine $\mathcal R$ which, over any 
input string $\alpha$ coinciding with $\mathsf{newell}(E)$ for  some 
$E\in \mathcal E$,  outputs the string
  $\mathcal R(\alpha)=\mathsf{ell} (E)$. 
Conversely, there is a Turing machine $\mathcal S$ which, over any 
 input string   $\beta=\mathsf{ell} (F)$ for some $F\in\mathcal E$,  
 outputs the string
 $\mathcal S(\beta)=\mathsf{newell}(F)$.
The  two maps  $\alpha\mapsto \mathcal R(\alpha)$ 
 and $\beta \mapsto \mathcal S(\beta)$ are inverses of each other.
 \end{proposition}
 

\section{Polyhedra in  $ \GLnZ$ geometry}
\label{section:polyhedra} 
 
Following \cite[1.1]{sta}, by a
  {\it polyhedron}  (``compact polyhedron'',  in \cite[2.2]{rousan})  
  we mean the union $P=\bigcup_iS_i$ of finitely many
  simplexes  $S_i$   in $\mathbb R^n$. 
$P$ need not be convex or connected. 
The $S_i$ need not have 
the same dimension.
If the vertices of each $S_i$ have  rational coordinates,
$P$ is said to be a {\it rational polyhedron}.


Rational polyhedra  play a key  role in
the  recognition  problem of  combinatorial
manifolds  presented  
as   rational polyhedra  $X,Y$. 
As a matter of fact,  (see, e.g., \cite[p.55]{glamad}),
$X$ is homeomorphic to $Y$ iff
there is a 
 {\it rational PL-homeomorphism} $\eta$ 
 of $X$ onto $Y$, i.e.,
 a  finitely piecewise affine linear  (PL) one-one
 continuous  map
  $\phi$  of $X$ onto $Y$  such that every affine linear piece of  
$\phi$   has rational coefficients.


It follows that the set $ \mathcal S$ of 
pairs of  rationally PL-homeomorphic
 polyhedra is  recursively
enumerable.  And yet,  the complementary set  is not:

\medskip 
\begin{theorem}
 [A.A. Markov, 1958. 
See  \cite{chelek, sht} and references therein] 
\label{theorem:markov}
The problem whether two rational polyhedra  
$X$  and $Y$   are rationally PL-ho\-m\-eo\-mor\-ph\-ic is 
   undecidable.
   \end{theorem} 
 
\medskip
While  stated in terms of  rational polyhedra, this
theorem  had enough impact to
put an end to the  (Klein)  program of attaching
to {\it any} combinatorial manifold $X$ 
an  invariant characterizing $X$ up to homeomorphism.

It is an interesting open problem whether Markov unrecognizability 
theorem 
still holds when
rational PL-homeomorphisms
are replaced by  integer PL-ho\-m\-e\-o\-morph\-i\-sms. 
%

\smallskip
Given rational polyhedra $P,P'\subseteq \mathbb R^n$ let
us agree so write $P\cong P'$ if some $\gamma\in \GLnZ$
maps $P$ onto $P'$, in symbols, $\gamma\colon P\cong P'.$

 \begin{figure} 
    \begin{center}                                     
    \includegraphics[height=8.9cm]{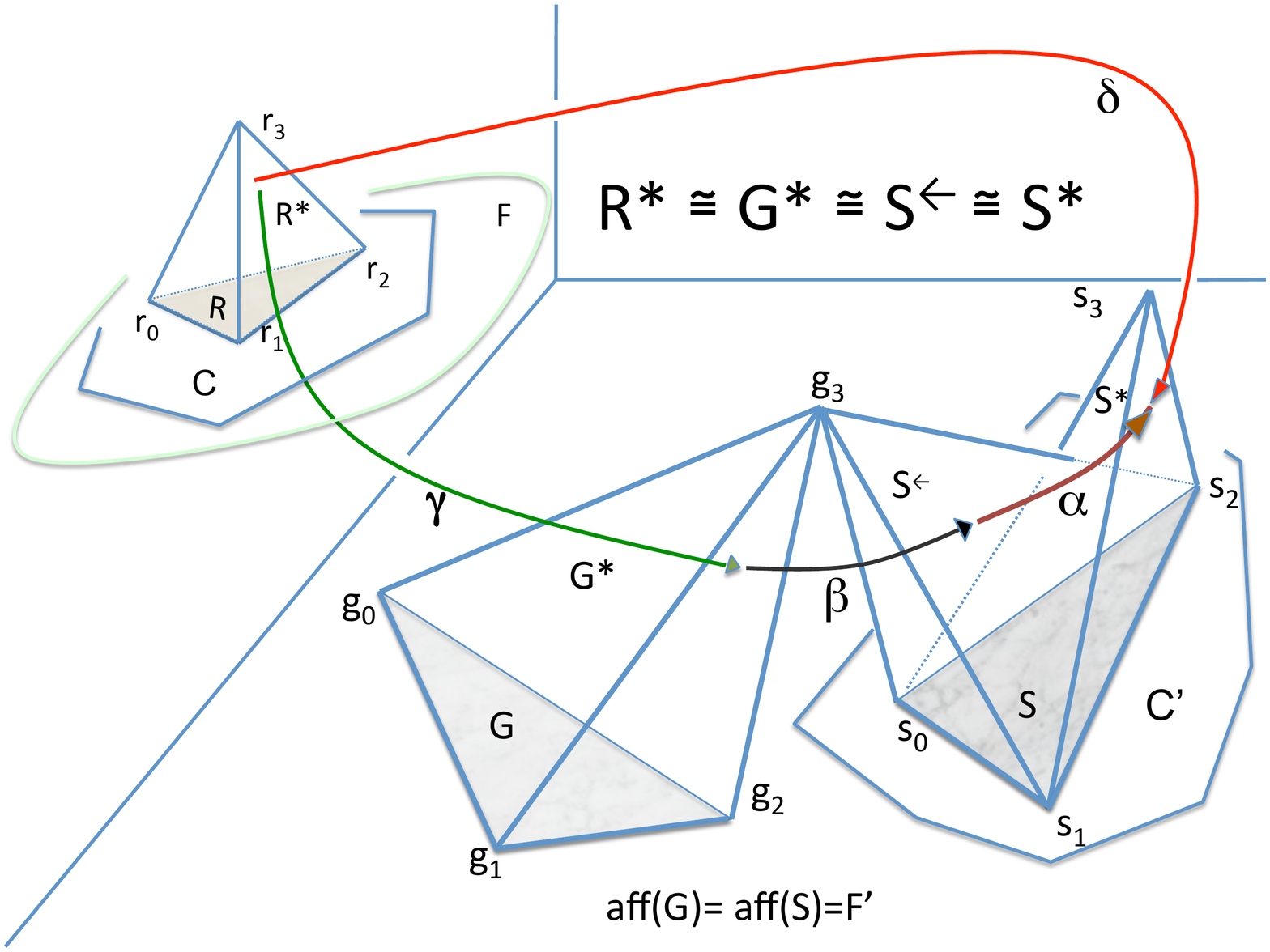}    
    \end{center}                                       
 \caption{\small  The special case $n=3$, $e=2$ 
 of the proof of Theorem 
 \ref{theorem:polyhedra}.}  
    \label{figure:tretetra}                                                 
   \end{figure}
 \begin{theorem}
 [Recognizing  rational polyhedra in $\GLnZ$-geometry]
\label{theorem:polyhedra}
The following problem is      decidable:

\medskip
\noindent
${\mathsf{INSTANCE}:}$    
Rational polyhedra $P=\bigcup_{i=1}^l S_i$ 
and $P'=\bigcup_{j=1}^m T_j\,,$ 
where each  $S_i$ and $T_j$ is a rational simplex
 in $\mathbb R^n$,
presented by the list of its vertices.
 
\medskip
\noindent
${\mathsf{QUESTION}:}$  
Does there exist  $\delta \in \GLnZ$ such that $\delta(P)=P'$?

\smallskip 
\noindent
Moreover, whenever such $\delta$ exists it can be
effectively computed.
\end{theorem}

\begin{proof}
From the two  lists of simplexes $S_i, T_j$ we 
construct rational triangulations
$\nabla$ of $P$ and $\nabla'$ of
 $P'$ following   \cite[Chapter II]{sta}.  
 (Also see  \cite[\S\S 17 and 25]{hand}.)
From the  vertices of 
  $\nabla$ and $\nabla'$, the algorithmic procedure of 
  {\cite[\S\S 7 and  22]{hand}}
yields the   (vertices of the) convex hulls
  $
  C= \conv(P) \,\,\,\mbox{ and } \,\,\,
  C'  =\conv(P')\subseteq \mathbb R^n.
  $
Let 
$$
F=\aff(P)=\aff(C)\,\,\, \mbox{ and } \,\,\,F'=\aff(P')=\aff(C').
$$
If   $P\cong P'$ then $C\cong C'$ and $F\cong F'$.  %
 Using the decision procedure of Corollary 
 \ref{corollary:affine-subspace}
   we  check whether 
the affine subspaces $F$ and $F'$
have the same $\GLnZ$-orbit. 
If this condition fails, our problem has a negative
answer. Otherwise,  we  introduce the notation  
  $$\dim(C)=\dim(C')=e,\,\,\,
  d_{F}=d_{F'}=d, \,\,\,
 c_{F}=c_{F'}=c.$$   
 Corollary \ref{corollary:affine-subspace}  yields 
 a map $\gamma\in \GLnZ$ such that 
\begin{equation}
\label{equation:hull}
\gamma\colon F\cong F'.
\end{equation}
 Since the rational polyhedron  $C$ is $e$-dimensional and convex,
 we  fix, once and for all,
  rational points $r_0,\dots,r_e\in C$ and
  additional points  $r_{e+1},\dots,r_n\in 
  \mathbb Q^n$
 such that $\conv(r_0,\dots,r_n)$ is an 
 $n$-simplex in $\mathbb R^n$. 
 Using, if necessary, the desingularization procedure
 described in 
  \cite[VI, 8.5]{ewa} or \cite[p.48]{ful},
 we may safely
 assume that $\conv(r_0,\dots,r_n)$  is regular.
 Let us use the abbreviations
 \begin{equation}
 \label{equation:V}
 V=(r_0,\dots,r_n)\in ( \mathbb Q^n)^{n+1},
 \,\, R=\conv(r_0,\dots,r_e), \,\,\,R^*=\conv(r_0,\dots,r_n).
\end{equation}
For  each  $i=0,\dots,n$ let us  set 
 $$
 g_i=\gamma(r_i),\,\,\, G=\conv(g_0,\dots,g_e) \mbox{ and }
 G^*=\conv(g_0,\dots,g_n).
 $$
 Observe that $G^*=\gamma(R^*)$ is a regular $n$-simplex, 
 $G=\gamma(R)$ is  a regular    $e$-simplex, and\,\, $\aff(G)$\,\,
 coincides with $ F'$ by \eqref{equation:hull}.

\bigskip
\noindent{\it Claim.}  The following conditions are equivalent:

\begin{itemize}
\item[(I)]
  There exists a map $\delta \in \GLnZ$
of $P$ onto $P'$.  

\smallskip
\item[(II)] There are    rational points
$s_0,\dots,s_e\in C'$  with the following properties:

\medskip
(i)  $\den(s_i)=\den(r_i)$ for each  $i=0,\dots,e$.

\medskip
(ii) $\conv(s_0,\dots,s_e)$ is a regular $e$-simplex;
(Thus  by
Lemma \ref{lemma:steinitz}
 the $n$-simplex  
$\conv(s_0,\dots,s_e,g_{e+1},\dots,g_n)$ is regular,
because
$G^*$ is regular  and  
$\aff(\{s_0,\dots,s_e\})=\aff(C')=\aff(G)=F'$).

\medskip

(iii)  Letting  $W\subseteq  ( \mathbb Q^n)^{n+1}$
be defined by 
\begin{equation}
\label{equation:W}
W=(s_0,\dots,s_e,g_{e+1},\dots, g_n),
\end{equation}
the map   $\phi=\phi_{VW}$ of 
$R^*$ onto $\conv(W)$
 given by Lemma \ref{lemma:affinemap}  sends $P$ onto $P'$.
 (Lemma \ref{lemma:affinemap}
 can be applied because  both  simplexes
$R^*$
and 
$\conv(W)$
 are regular and the denominators of their
vertices are pairwise equal.)
 \end{itemize}

 \bigskip
For the nontrivial  direction,
 suppose some $\delta\in \GLnZ$ maps $P$ onto $P'$.
 Then  
 $\delta\colon C\cong C'$ and $\delta\colon F\cong F'.$
For  each  $ i=0,\dots,n$  let us  define the rational  point $s_i\in C'$ by 
 \begin{equation}
 \label{equation:lorraine}
 s_i=\delta(r_i),
 \end{equation}
 with the intent of proving that $s_0,\dots,s_e$ 
 satisfy conditions   (i)-(iii).

\smallskip
 Condition  (i)   is immediately satisfied,
 because  $\delta$ preserves denominators.
Next, let us set 
  $S=\conv(s_0\dots,s_e)=\delta(R) \mbox{ and } 
  S^*=\conv(s_0\dots,s_n)=\delta(R^*).$
 Since $S^*$  is  regular    then so is     $S$, 
 and condition (ii) is satisfied. There remains
 to be proved that $s_0,\dots,s_e$  satisfy condition (iii).
 To this purpose let us first note that 
   both  $\gamma$ and $\delta$ map  $F$ onto $F'$, and hence
$
\aff(S)=\aff(G)=F'.
$
Further, from
$
\gamma\colon R^*\cong G^*
\,\,\,\mbox{and}\,\,\,\delta\colon R^*\cong  S^*
$
we get  
$S^*\cong G^*.$
By restriction, we obtain  regular $e$-simplexes
$R\cong G\cong S$
having the same $\GLnZ$-orbit.
Let
$
S^{\leftarrow}= \conv(s_0,\dots,s_e, g_{e+1},\dots,g_n).
$
Since
 the vertices  $g_{e+1},\dots,g_n$ are common to both 
$G^*$ and $S^{\leftarrow}$, by \eqref{equation:lorraine} we have
\begin{equation}
\label{equation:rs} 
\den(s_i)=\den(r_i)\,\,\,\mbox{for all } i=0,\dots,n.
\end{equation}
Since  $G^*$ is a regular $n$-simplex,   
 by Lemma \ref{lemma:steinitz} 
so is $S^{\leftarrow}$. 
Therefore, by
  Lemma \ref{lemma:affinemap},
  there is a  uniquely determined
$ \beta\in \GLnZ$  such that  
$
\beta\colon G^*\cong S^{\leftarrow}.
$
The two
$n$-simplexes  $S^{\leftarrow}$
and $S^*$ are regular and their vertices have pairwise equal
denominators,  because their first $e+1$ vertices
$s_0,\dots,s_e$ coincide, and  by  \eqref{equation:rs},
$$
\den(g_{i})=\den(r_i) =\den(s_i) \mbox{ for all }  i=e+1\dots,n.
$$

\noindent
Another application of  Lemma \ref{lemma:affinemap} 
yields  a uniquely determined
$\alpha\in \GLnZ$  such that 
$
\alpha\colon S^{\leftarrow}\cong S^*.$
It follows that 
$
\delta=\alpha\circ\beta\circ\gamma,
$
where  ``$\circ$'' denotes composition. 
Let  $\phi=\beta\circ \gamma.$ 
Then  $\phi\in \GLnZ$ maps  $R^*$ onto $S^{\leftarrow}$. 
Specifically, recalling  \eqref{equation:V}-\eqref{equation:W},
$\phi$ coincides with 
the map $\phi_{VW}$ of  
Lemma \ref{lemma:affinemap}.
By construction,  $\phi$
agrees with $\delta$ over $R$  (whence  $\phi$ agrees with
$\delta$ over $F\supseteq C \supseteq P$). 
Since $\delta$ maps $P$ onto $P'$,
then so does $\phi.$ Thus the points
$s_0,\dots,s_e$ satisfy  condition (iii). Our claim is settled.

 \medskip
Let now $\Omega$ be the set  of all $(e+1)$-tuples 
  $s=(s_0,\dots,s_e)$ of rational 
points  in $C'$ such that 
$\den(s_i)=\den(r_i)$ for all $i=0,\dots,e$, and  the set
$T_s=\conv(s_0,\dots,s_e)$ is a regular 
$e$-simplex---a condition that
can be  effectively checked.  
 $\Omega$ is a finite set,
because 
$C'$ is  bounded  and there are only finitely many
rational points in $C'$ with denominators  
$\leq \max(\den(r_0),\dots,\den(r_e)).$
It is easy to see that
   $\Omega$ is the set of all $(e+1)$-tuples
of rational points in $C'$ satisfying conditions (i)-(ii) in our claim.
Letting  the $n$-simplex $T^{\leftarrow}_s$ be defined  by 
$T^{\leftarrow}_s=\conv(s_0,\dots,s_e, g_{e+1},\dots,g_n),$  the
  regularity of $T^{\leftarrow}_s$ follows
   from the regularity of $T_s$ and of  $G^*$,  
   by  Lemma \ref{lemma:steinitz}.
From  the  $(n+1)$-tuple of rational points
$V$ defined in \eqref{equation:V} and the  $(n+1)$-tuple
 $
   U(s)=(s_0,\dots,s_e, g_{e+1},\dots,g_n), 
 $
  Lemma \ref{lemma:affinemap}  
yields  a uniquely determined
map $\phi_{VU(s)}\in \GLnZ$ of $R^*$ onto $T^{\leftarrow}_s$.

By  our claim, 
  $P\cong P'$ iff 
for at least one
$(e+1)$-tuple  $\bar s=(\bar s_0,\dots, \bar s_e)\in \Omega$,
 $\phi_{VU(\bar  s)}$  maps  
 $P$ onto $P'$, i.e., $\bar  s$ also satisfies
 condition (iii).  This final condition is decidable, by 
 resorting to   the triangulations
 $\nabla$ and $\nabla'$  constructed at the outset of the proof:
 indeed, we must only check whether  for each  
simplex in $ \nabla$
its $\phi_{VU(\bar  s)}$-image 
 is contained in the  union of simplexes of
$\nabla',$  and vice-versa, check  whether 
for each  
simplex in $ \nabla'$
its $\phi_{VU(\bar  s)}^{-1}$-image 
 is contained in the union of simplexes of
$\nabla.$

We have just shown the decidability of the problem
whether  there is   a map $\delta \in \GLnZ$
of $P$ onto $P'$. Our constructive proof also shows that
 whenever any such  $\delta$
exists, 
   it can be
effectively computed. 
\end{proof}

 \noindent
   Figure    \ref{figure:tretetra} illustrates the
   crux of the proof for  
     $n=3$ and $e=2$.

\subsection*{Concluding Remarks}
\label{section:final}

Suppose  $P$ and $Q$ are
finite unions of $n$-dimensional rational simplexes
 in $\mathbb R^n$
(for short,  $P$ and $Q$ are rational  ``$n$-polyhedra'').
Suppose there are rational triangulations $\Delta$ of 
$P$ and $\nabla$ of
$Q$ such that every simplex $T$ of $\Delta$ can be mapped
one-one onto a simplex of $\nabla$ by some 
$\eta_T \in \GLnQ$, in such a way that the set
$\eta=\bigcup\{\eta_T \mid T \in \Delta\}$ is a continuous one-one
map. We then say  that 
$\eta$ is a ``continuous $\GLnQ$-equidissection'', 
\cite[31.3]{hand}.

Markov unrecognizability 
 theorem 
  is to the effect  that the continuous
$\GLnQ$-equidissectability of $P$ and $Q$ is not decidable. 
It is an interesting open problem whether Markov's  theorem
still holds when
continuous $\GLnQ$-equidissections are
  replaced by  
 continuous $\GLnZ$-equidissections.
The subproblem of  
deciding whether $P$
 and $Q$
have  the same    $\GLnZ$-orbit has been shown to
be decidable in Theorem \ref{theorem:polyhedra}. 

%

%
%
%
%

Differently from the case of angles, segments,
triangles and ellipses,  our positive solution of Problem  
\eqref{equation:theproblem} for 
 rational polyhedra does  not rest on the
 assignment of  a  computable  complete
invariant to every
 rational polyhedron  $P\subseteq \mathbb R^n$.
 And yet, 
 $P$  is  equipped with a wealth of
 computable invariants for  
 continuous $\GLnZ$-equidissections---well beyond
 the classical homeomorphism invariants given by   dimension, 
  number of connected components,  or Euler characteristic.
These invariants include:
The number  of
 rational points in $P$ of a given denominator $d=1,2,\dots$;
The number  of regular  triangulations $\Delta$ of $P$
  such that the denominators of all vertices of $\Delta$ are $\leq d;$ \,\,
The smallest possible number of
 $k$-simplexes in a regular triangulation $\Delta$
 of $P$ such  that
 the denominators of all vertices of $\Delta$ are $\leq d$.
 All these new invariants are  (a fortiori), 
$\GLnZ$-orbit invariants---and none  makes  sense in euclidean
 geometry,  or even  in $\GLnQ$-geometry.  
 Closing a circle of ideas,  
one may then naturally
ask the following question:

\subsection*{Problem  ($n=2,3,\dots$)}\,\,\,
Can  the $\GLnZ$-orbit problem for    
  rational $n$-polyhedra
be decided  by 
computable  complete invariants?

\end{document}